\documentclass[11pt,reqno]{amsart}\usepackage{tikz}
\usepackage{amssymb,amsbsy,amsmath,amsfonts,amssymb,amscd}
\usepackage{mathrsfs,makecell}
\usepackage{epsfig, graphicx}
\usepackage{graphics,epsfig}
\usepackage{mathrsfs}
\usepackage{amsthm}
\usepackage{mathtools}
\usepackage{graphicx,amsmath,amsfonts,amsthm,amssymb,algorithm,algpseudocode,xcolor} 
\newtheorem{theorem}{Theorem}[section]
\newtheorem{lemma}[theorem]{Lemma}
\DeclareMathOperator*{\argmin}{arg\,min}
\newcommand{\blue}[1]{{\color{black}{#1}}}

\graphicspath{{"Figs/"}}
\usepackage[margin=1.8cm]{geometry}
\newcommand{\tP}[1]{{\tilde{P}}}
\DeclareMathOperator{\Tr}{Tr}
\newcommand{\ts}{\mathsf{T}}

\begin{document}
 
 \title[Tensor train based sampling algorithms]{Tensor train based sampling algorithms for approximating regularized Wasserstein proximal operators
}
\date{}
\author[Han]{Fuqun Han}\address{Department of Mathematics, University of California, Los Angeles}
\email{fqhan@math.ucla.edu}
\author[Osher]{Stanley Osher}
\address{Department of Mathematics, University of California, Los Angeles,
Los Angeles, CA, USA.}
\email{sjo@math.ucla.edu}
\thanks{F. Han and S. Osher are partially funded by AFOSR MURI FA9550-18-502 and ONR N00014-20-1-2787}
\author[Li]{Wuchen Li}
\address{Department of Mathematics, University of South Carolina, Columbia, SC, USA.}
\email{wuchen@mailbox.sc.edu}
\thanks{W. Li’s work is partially supported by AFOSR YIP award No.FA9550-23-1-0087, NSF DMS-2245097, and NSF RTG: 2038080.}

\begin{abstract}
We present a tensor train (TT) based algorithm designed for sampling from a target distribution and employ TT approximation to capture the high-dimensional probability density evolution of overdamped Langevin dynamics. This involves utilizing the regularized Wasserstein proximal operator, which exhibits a simple kernel integration formulation, i.e., the softmax formula of the traditional proximal operator. The integration, performed in $\mathbb{R}^d$, poses a challenge in practical scenarios, making the algorithm practically implementable only with the aid of TT approximation.
In the specific context of Gaussian distributions, we rigorously establish the unbiasedness and linear convergence of our sampling algorithm towards the target distribution. 
To assess the effectiveness of our proposed methods, we apply them to various scenarios, including Gaussian families, Gaussian mixtures, bimodal distributions, and Bayesian inverse problems in numerical examples. The sampling algorithm exhibits superior accuracy and faster convergence when compared to classical Langevin dynamics-type sampling algorithms.
\end{abstract}
\maketitle


\section{Introduction}
Over the past decade, obtaining samples from a known and potentially complicated distribution has become increasingly vital in the fields of data science, computational mathematics, and engineering. Sampling algorithms play a central role in many critical real-world applications, including finding global optimizers for a high-dimensional function  \cite{sampling_optimization_2019_MJ}, obtaining samples from the latent space in generative modeling \cite{song2019generative}, and solving Bayesian inverse problems to estimate the posterior distribution \cite{stuart2010inverse}. Efficient and reliable sampling algorithms are essential for the success of the aforementioned applications.


Given the significance of sampling from a known distribution, numerous intriguing algorithms have been proposed and analyzed. Among them, Markov Chain Monte Carlo (MCMC) type algorithms have been the most popular due to their simple formulation and intrinsic diffusion resulting from the adoption of Gaussian noise, which is desirable. Representative MCMC type algorithms include Metropolis random walk \cite{random_work_Metroplis}, hit and run \cite{hit_and_run} which are zero-order methods, and unadjusted Langevin \cite{ULA_first}, Hamiltonian Monte Carlo \cite{mackay2003book}, Metropolis-adjusted Langevin algorithms
 \cite{early_MALA}, which are first-order methods using Langevin diffusions. For a more in-depth review and details for Monte Carlo type algorithms, one may refer to \cite{brooks2011handbook} and references therein. Langevin-type MCMC methods usually involve evaluating the gradient of the potential function and adding a Gaussian noise to achieve diffusion. Several theoretical results showed that they can converge to the stationary distribution under proper assumptions \cite{ULA_2019_convergence,MALA_2019_converge}. \blue{
Moreover, the Langevin dynamics-based particle evolution has been successfully applied to Bayesian inference, which we will also explore in our numerical experiments, for instance \cite{FKP_Bayersian}.}
 
Classical Langevin dynamics-type sampling algorithms may exhibit slow convergence, particularly for complex and high-dimensional distributions. As an alternative approach, generating diffusion using the score function, defined as the logarithm of the density function, offers a promising direction for sampling algorithms. In \cite{song2019generative}, the author demonstrates that score-based diffusion minimizes the Kullback–Leibler divergence (KL divergence) between the target and generated distribution, a perspective that can also be interpreted as an entropy-regularized optimal transport problem \cite{kwon2022score}.
Score-based diffusion has already showcased impressive numerical performance in various domains, including image generation \cite{dhariwal2021diffusion}, medical inverse problems \cite{chung2022score}, and importance sampling \cite{doucet2022score}. Despite these successes, a notable challenge in many score function based algorithms lies in efficiently approximating the score function, stemming from the inherent complexity associated with solving the density function.


Recently, in \cite{BRWP_2023}, a new score function based sampling algorithm via the backward Wasserstein proximal operator (BRWP) is proposed. In that work, the authors approximate the score function by considering a regularized Wasserstein proximal, which can be shown to have a kernel formula through the Hopf-Cole transform. With the explicit form for the computation of the score function, the sampling algorithm becomes deterministic and relatively robust. Numerical experiments show that the evolution of samples exhibits a highly structured manner and a reliable convergence is guaranteed for several interesting scenarios. However, as the author pointed out in \cite{BRWP_2023}, the algorithm is biased due to the discretization of the particle evolution equation and is not easily scalable to high-dimensional spatial domains.


To address the curse of dimensionality inherent in high-dimensional sampling problems, many strategies have been employed. These approaches encompass the investigation of low-rank structures \cite{cui2016dimension, wang2022projected}, the utilization of accelerated gradient flow techniques \cite{Nesterov_MCMC, wang2022accelerated}, and the incorporation of deterministic sparse quadrature rules tailored for special cases \cite{schillings2013sparse, gantner2016computational}. In this work, we shall leverage efficient and delicate tensor algorithms to improve both theoretical results and the numerical performance of BRWP. Tensor train approximation, which can represent a broad class of functions with storage complexity $\mathcal{O}(d)$ and allows many algebraic computations of order  $\mathcal{O}(d)$, where $d$ is the dimension, has become popular in recent years. Tensor train approximations and algorithms, also called \blue{matrix product states} \cite{TT_Physics_origin}, are populated in \cite{TTSVD} and find applications in neural network representation \cite{TT_network}, achieving image super-resolution \cite{TT_image}, high-dimensional PDE simulation \cite{TT_PDE_simulation}, generative modeling \cite{TT_Gen_mod}, uncertainty quantification \cite{Eigel_2023} and so on. \blue{Many of the aforementioned applications rely on advancements such as the continuous tensor train approximation \cite{TT_cts}, the solution of linear systems in tensor train format \cite{AMEN_TT}, the development of crossing algorithms \cite{tt_cross}, and so on.} In this paper, we fully harness the power of tensor train approximation to derive proper tensor sampling algorithms that are computationally efficient and accurate for high-dimensional distributions. 

The proposed sampling algorithm presents several distinctive features. Firstly, from a computational complexity perspective, since the Gaussian kernel used in the score function approximation can be expressed as a tensor train of rank 1, our tensor train approach evaluates the kernel formula with a complexity of only \( \mathcal{O}(rNd) \), where \( r \) depends on the target distribution and \( N \) is the number of points per dimension. In contrast, direct quadrature requires \( \mathcal{O}(N^d) \) operations to achieve the same accuracy. Furthermore, the deterministic tensor-based approximation typically provides more accurate and stable results compared to Monte Carlo estimation.
Secondly, the paper introduces a delicate choice of covariance matrix for initial density estimation through the discretization of the density function on a high-dimensional mesh, leading to an unbiased sampling algorithm. Thirdly, diffusion, which is generated by a deterministic approximation procedure of the score function endows the proposed method with great robustness for highly ill-posed Bayesian inverse problems.

The paper is structured as follows. Section \ref{sec_Problem_Stat} provides an insightful review of the score function based sampling algorithm focusing on the BRWP proposed in \cite{BRWP_2023} and emphasizing the kernel representation of the solution for the regularized Wasserstein proximal operator. Section \ref{sec_TT_Basics} will introduce the fundamental concept of tensor train approximation and will present several important motivations for the adoption of tensor train approximation for our objective. Section \ref{sec_analysis} intricately details the proposed tensor train based sampling algorithm, offering comprehensive insights into its implementation and convergence analysis across critical scenarios. Finally, Section \ref{sec_NE} presents a series of numerical experiments to validate the proposed sampling algorithm's claimed features and provide a robust empirical foundation for its efficacy in real-world applications, such as Bayesian inverse problems. 

\section{Problem Statement and Score-Based Sampling Algorithm}
\label{sec_Problem_Stat}

In this section, our main goal is to present the definition of the score function and outline an approximation problem for a regularized Wasserstein proximal operator. The kernel formula for this operator is key to the derivation of our sampling algorithm. Following this, we provide a brief description of the Backward Regularized Wasserstein Proximal Scheme (BRWP) proposed in \cite{BRWP_2023}, which serves as motivation for the current work.
 
Our specific objective is to acquire a series of samples $\{x_{k,j}\}_{j=1}^N$ from a distribution $\Pi_k$, approximating the target distribution $\Pi^*$ with a density function
\begin{equation}
\rho^*(x) = \frac{1}{Z}\exp\left(-\beta V(x)\right),
\end{equation}
where $V(x)$ is a known continuously differentiable potential function, and \[Z=\int_{\mathbb{R}^d}\exp(-V(y))dy<+\infty\] is a normalization constant.

Score-based diffusion concerns the evolution of particles with density function $\rho$, following the Langevin stochastic differential equation
\begin{equation}
\label{SDE}
dX(t) = -\nabla V(X(t)) dt + \sqrt{2/\beta} dW(t), \quad X(0) = X_0,,
\end{equation}
where $\beta>0$ is a constant, $W(t)$ is the standard Wiener process in $\mathbb{R}^d$, and $X_0$ is the initial set of particles.

With the help of Louisville's equation \cite{FP_score}, we derive the scored-based particle evolution equation whose trajectories have the same marginal distribution as \eqref{SDE}
\begin{equation}
\label{particle_evolution}
\frac{dX}{dt} = -\nabla V(X) - \beta^{-1} \nabla \log\rho(t,X),
\end{equation}
where $\nabla \log\rho(t,X)$ is the score function associated with density $\rho$.

To compute the density function, it is known that $\rho$ will evolve following the Fokker-Planck equation
 \begin{equation}
 \label{FKP_eq}
    \frac{\partial \rho}{\partial t} = \nabla\cdot (\nabla V(x)\rho)+\beta^{-1}\Delta\rho\,,\quad \rho(0,x) = \rho_0(x)\,, 
\end{equation}
for the initial density $\rho_0$. 

Approximating the terminal density $\rho_T:= \rho(T,\cdot)$ with terminal time $T$ is challenging due to non-linearity and high dimensions. To address this, we consider a kernel formula to approximate a regularized Wasserstein operator. Specifically, we first recall the Wasserstein proximal with linear energy
\begin{equation}
\label{Wass_linear_energy}
   \rho_T=\arg\min_{q} \bigg[\frac{1}{2T}W(\rho_0, q)^2+ \int_{\mathbb{R}^d}V(x)q(x) dx\bigg]\,, 
\end{equation}
where $W(\rho_0, q)$ is the Wasserstein-2 distance between $\rho_0$ and $q$. Additionally, by the Benamou-Brenier formula, the Wasserstein-2 distance can be expressed as an optimal transport problem, leading to
\[
\frac{W(\rho_0, q)^2}{2T} = \inf_{\substack{\rho, v, \rho_T}} \int_{0}^{T} \int_{\mathbb{R}^d} \frac{1}{2} \|v(t, x)\|_2^2 \rho(t, x) \,dx\,dt,
\]
where the minimizer is taken over all vector fields $v : [0, T] \times \mathbb{R}^d \to \mathbb{R}^d$, density functions $\rho : [0, T) \times \mathbb{R}^d \to \mathbb{R}$, such that
 \begin{align*}
&\frac{\partial \rho}{\partial t} + \nabla \cdot (\rho v) = 0\,,\quad \rho(0, x) = \rho_0(x)\,,\quad \rho(T, x) = q(x)\,.
\end{align*}

 Solving \eqref{Wass_linear_energy} is usually a challenging optimization problem. Motivated by Schr\"{o}dinger bridge systems, in \cite{kernel_proximal}, the authors introduce a regularized Wasserstein proximal operator by adding a Laplacian regularization term, leading to
\blue{ \begin{equation}
\label{Reg_Was}
   \rho_{T} = \argmin_{q}\inf_{v,\rho}  \int_{0}^{T} \int_{\mathbb{R}^d} \frac{1}{2} |\!|v(t, x)|\!|^2 \rho(t, x) \,dx\,dt + \int_{\mathbb{R}^d} V(x) q(x) \,dx\,,
\end{equation}}
with
\begin{equation}
   \frac{\partial \rho}{\partial t}  + \nabla \cdot (\rho v) = \beta^{-1} \Delta \rho,, \quad
    \rho(0, x) = \rho_0(x)\,, \quad \rho(T, x) = q(x)\,,
\end{equation}
where $\rho(t,x)$ as the solution to the above regularized Wasserstein proximal operator approximates one step of the Fokker-Planck equation when $T$ is small which can be seen in the next equation. 

Introducing a Lagrange multiplier function $\Phi$, we find that solving $\rho_T$ is equivalent to computing the solution of the coupled PDEs
\begin{align}
\label{regu_PDE}
\begin{cases}
    &\partial_t \rho + \nabla_x\cdot (\rho \nabla_x\Phi) = \beta^{-1} \Delta_x\rho\,, \\
    &\partial_t \Phi + \frac{1}{2}||\nabla_x\Phi||^2 = -\beta^{-1} \Delta_x\Phi\,, \\
    &\rho(0,x) = \rho_0(x)\,, \quad \Phi(T,x) = -V(x)\,.
\end{cases}
\end{align}
\blue{
Comparing the first equation in \eqref{regu_PDE} with the Fokker-Planck equation defined in \eqref{FKP_eq}, we observe that the solution to the regularized Wasserstein proximal operator \eqref{Reg_Was} approximates the terminal density \(\rho\) when \(T\) is small. The motivation for considering the regularized Wasserstein proximal operator as an approximate solution lies in the fact that the coupled PDEs can be solved using a Hopf-Cole type transformation.} Then $\rho(t,x)$ can be computed by solving a system of backward-forward heat equations 
\begin{equation}
\label{Heat_rho}
    \rho(t,x) = \eta(t,x)\hat{\eta}(t,x)\,,
\end{equation}
where $\hat{\eta}$ and $\eta$ satisfy
\begin{align}
&\partial_t\hat{\eta}(t,x) = \beta^{-1} \Delta_x\hat{\eta}(t,x)\,, \label{backward_heat}\\
&\partial_t\eta(t,x) = -\beta^{-1} \Delta_x\eta(t,x)\,, \label{forward_heat}\\
&\eta(0,x)\hat{\eta}(0,x) = \rho_0(x), \quad \eta(T,x) = \exp\left(-\beta V(x)\right)\,.
\end{align}

\blue{
In summary, equation \eqref{Heat_rho} provides a closed-form update to compute the density function by solving a system of backward-forward heat equations, which can be explicitly evaluated using heat kernels. Once \(\rho_T\) is determined, the evolution of particles in the next iteration can be computed using the following discretization for equation \eqref{particle_evolution}:
\begin{equation}
\label{FP_Euler}
    X_{k+1} = X_k - h\big(\nabla V(X_k) + \beta^{-1}\nabla\log(\rho_T(X_k))\big),
\end{equation}
where \(h\) is the step size and \(\rho_T\) is the terminal density generated from initial density $\rho_k$.
We remark that our particle evolution equation can be considered as a 'semi-backward' Euler discretization of the overdamped Langevin dynamics, as the score function is evaluated at time \(t_k + T\) at the \(k\)-th step.}

For a comprehensive understanding of our upcoming discussion, we recall the key steps in the BRWP algorithm developed in \cite{BRWP_2023}. This algorithm estimates the initial density $\rho_0$ by empirical distribution, computes \eqref{Heat_rho} by convolution with a heat kernel as
\begin{equation}
\label{rho_T_BRWP}
 \rho_T(x) =  \int_{\mathbb{R}^d}\frac{\exp\big[-\frac{\beta}{2}\big(V(x)+\frac{||x-y||_2^2}{2T}\big)\big]}{\int_{\mathbb{R}^d}\exp\big[-\frac{\beta}{2}\big(V(z)+\frac{||z-y||_2^2}{2T}\big)\big]dz}\rho_0(y)dy
\end{equation}
and generates new samples by \eqref{FP_Euler}. In particular, the integral in $\mathbb{R}^d$ that appeared in \eqref{rho_T_BRWP} is estimated using Monte-Carlo integration.


\section{Review on Tensor Train Approximation and Algorithms}
\label{sec_TT_Basics}

In this section, we will provide a concise overview of the definition, algorithms, and convergence properties of tensor train (TT) approximation applied to several crucial classes of multivariate functions of interest. The general objective of TT approximation is to reduce the computational complexity and improve the robustness of the sampling algorithms that we will introduce in the subsequent sections.

The TT decomposition $f_{TT}$ of a $d$-dimensional tensor $f: \mathbb{R}^d\rightarrow \mathbb{R}$ is defined as follows:
\begin{equation}
\label{def_TT}
  f(x_1, \ldots, x_d)  \approx f_{TT}(x_1, \ldots, x_d) = \sum_{\alpha_1=1}^{r_1} \cdots \sum_{\alpha_{d-1}=1}^{r_{d-1}} g_1(x_1, \alpha_1)g_2(\alpha_1, x_2, \alpha_2)  \cdots g_d(\alpha_{d-1}, x_d)\,,
\end{equation}
where $r = \max\{r_j\}_{j=1}^d$ is the rank of the decomposition, and the functions $g_i$ are cores of the TT decomposition. We remark \eqref{def_TT} can be generalized to functional approximation in $\mathbb{R}^d$.

The TT representation of a high-dimensional tensor enables efficient computation of various algorithms, including addition, Hadamard product, and matrix-vector products, especially when the rank $r$ is relatively small. Assuming the number of nodal points in each dimension is $n$, then the complexity of several frequently used algorithms for tensors in TT format is summarized in the following table \cite{TTSVD}:
\begin{center}
\begin{tabular}{|c|c|c|c|c|}
 \hline
 \text{Operation} & Rounding & Addition & Hadamard Product & Matrix-Vector Product \\
 \hline
 Complexity & $\mathcal{O}(dnr^3)$ & $\mathcal{O}(dnr^3)$ & $\mathcal{O}(dnr^3)$ & $\mathcal{O}(dn^2r^4)$ \\
 \hline
\end{tabular}
\end{center}

To use TT algorithms effectively, it is crucial to determine the class of functions for which the TT representation \eqref{def_TT} exists with a relatively small rank $r$. Since the TT decomposition can be viewed as a recursive singular value decomposition, we recall the following theorem from \cite{TT_accuracy} based on SVD if $f$ is a $d$-dimensional tensor.

\begin{theorem}
\label{Thm_TT_accuracy} \cite{TT_accuracy}
\blue{Let $f \in H^{k+1}(\mathbb{R}^d)$  for some fixed $k > 0$ and $0 < \epsilon < 1$. Then, the overall truncation error of the TT decomposition for $f$ with ranks $r \leq \epsilon^{-d/k}$ is given by}
\[
\| f - f_{TT}\|_{L^2(\mathbb{R}^d)} \leq \sqrt{d}\epsilon,
\]
and the storage cost for TT representation will be $\epsilon^{-d/k}$.
\end{theorem}

The above theorem is applicable since, as we will see later, for sampling problems, functions we approximate in TT format are the exponential of potential functions or simply the heat kernel, which is typically very smooth. In such cases, $k$ can be relatively large, resulting in small ranks for TT approximation and low storage complexity for a fixed $d$.

Moreover, it is noteworthy that the estimate in Theorem \ref{Thm_TT_accuracy} remains slightly unsatisfactory, considering that the storage cost and approximation error still depend on $\sqrt{d}$ or $\epsilon^{-d/k}$, which might be expensive for high-dimensional scenarios. For an alternative approximation result, specifically when $f$ represents the density function of a Gaussian distribution, we refer to \cite{rohrbach2020rank} which presents a convergence theorem that is independent of dimension and relies on the off-diagonal ranks of the variance matrix.

The next issue we need to address is the algorithm to obtain a proper TT decomposition efficiently for a given tensor. Based on our numerical experience, we found that the TT crossing algorithm in \cite{tt_cross} is the most efficient and stable one for our purpose. For a discretized tensor defined on $(\mathbb{R}^n)^d$, where $n$ is the number of nodal points in each dimension, the general idea of TT crossing-type algorithms is to unfold the tensor into a matrix $A_1$ of the form $\mathbb{R}^n\times \mathbb{R}^{n(d-1)}$, called the unfolding matrix. The standard crossing approximation is then applied on $A_1$ to have
\[
A_1 \approx C\hat{A}^{-1}R\,,
\]
with
\[
C = A_1(:,I_1)\,, \quad R = A_1(J_1,:)\,,
\]
where $I_1$ and $J_1$ are small index sets chosen by certain energy minimization algorithms. This process is repeated for each dimension. Notably, crossing algorithms only require access to a small number of entries of the original tensor, and the full tensor is never formed in the algorithm.

For the accuracy of the TT crossing algorithm, the following error bound of TT approximation in the Frobenius norm for $f$ discretized on $n$ points in each dimension is established in \cite{qin2023error}.

\begin{theorem} \cite{qin2023error}
\label{Thm_TT_cross}
    Suppose a tensor $f$ can be approximated by a tensor train with maximum rank $r$ and error $\epsilon$. Using the crossing algorithm with tolerance $\epsilon$, we can find $f_{TT}$ with rank at most $r$ such that
    \[
    ||f-f_{TT}||_F \leq \frac{(3\kappa)^{\log_2 d} -1}{3\kappa - 1}(r+1)\epsilon\,,
    \]
where $\kappa$ is the condition number of the unfolding matrix of $f$ \blue{whose precise definition can be found in \cite{qin2023error}.}
\end{theorem}

\subsection{A Straightforward Tensor Train Algorithm to Improve BRWP}
In this subsection section, we initially propose a straightforward sampling method that leverages TT algorithms discussed in the preceding subsection. This method is designed to compute high-dimensional integrals appearing in \eqref{rho_T_BRWP}. The error analysis of this algorithm offers valuable insights into the power of TT approximations in approximating the kernel formula, paving the way for the development of the new algorithm to be presented in the next section.

In the algorithm described in \cite{BRWP_2023}, for each iteration, it requires the computation of the density $\rho_T$ using the following formula
\begin{equation}
\label{empirical_rhoT}
    \rho_T(x_j) = \frac{1}{N}\sum_{i=1}^N \frac{\exp\left[-\beta/2 \left(\blue{V(x_j)} + ||x_i - x_j||_2^2/(2T)\right)\right]}{\int_{\mathbb{R}^d} \exp\left[-\beta/2 \left(V(z) + ||x_i - z||_2^2/2T\right)\right]dz}\,,
\end{equation}
for each particle located at $x_j$.

We observe that the bottleneck of computation in \eqref{empirical_rhoT} at each iteration lies in the normalization term
    \begin{equation}
\label{def_normal}
    N(x) := \int_{\mathbb{R}^d} \exp\left[-\frac{\beta}{2}\left(V(z) + \frac{||x - z||_2^2}{2T}\right)\right]dz\,,
\end{equation}
which involves integration in $\mathbb{R}^d$. Monte Carlo integration through random sampling can only provide a solution that converges at a rate of $N^{-1/2}$, where $N$ is the number of random points used in the integration. It is important to note that the $N^{-1/2}$ error bound does not imply that the computational complexity to achieve a fixed degree of accuracy is the same for any dimension, as the complexity of evaluating a $d$-dimensional function $f$ will also depend on $d$. For certain scenarios, as we will present in the following table \ref{table_TT_time}, the convergence could be extremely slow for some interesting high-dimensional integrals that arise in sampling problems.

In this case, we propose to apply TT approximation to improve both the accuracy and efficiency of the computation of the normalization term in \eqref{def_normal}. We first consider $n$ quadrature points $\{z_j\}_{j=1}^n \in [-L, L]$ in each dimension. The mesh formulated by $\bigotimes\{z_j\}$ in $[-L, L]^d$ is denoted as $\mathcal{Z}_{d,n}$. The integration then becomes
\begin{equation}
\label{exp_int}
    N(x) \approx K_d\mathcal{T}(\exp(-\beta V/2))(x),
\end{equation}
where $\mathcal{T}(f)$ is the tensor train approximation using the crossing algorithm for $f$ on the mesh $\mathcal{Z}_{d,n}$, and $K_{d}$ is defined as
\begin{equation}
\label{def_Kd}
    K_d(x, y) = K_1(x, y) \otimes \cdots \otimes K_1(x, y),
\end{equation}
with
\blue{\[
(K_1(x, y))_{j,k} = w_j \exp\left(-\beta\frac{||x_j -y_k||_2^2}{4T}\right),
\]
and $w_j$ being quadrature weights for $x_j$.
The error term \(\epsilon\) in the approximation \eqref{exp_int} consists of two components: the truncation error from limiting the integration to \([-L, L]^d\), which is small due to the exponential decay of \(\exp(-\beta||x - y||_2^2/4T)\) for large \(y\), and the quadrature error. The quadrature error will be of order \(n^{-n}\) if \(\{x_j\}\) and \(\{y_k\}\) are Legendre quadrature points for smooth functions.  
}

\blue{Then, the computation of the density function $\rho_T$ becomes
\begin{equation}
    \rho_T(x_j) = \int_{\mathbb{R}^d} G(x_j, y) \rho_0(y) dy\,, \quad G(x, y) = \mathcal{T}(\exp(-\beta V/2))(x)\frac{K_d(x, y)}{N(y)}\,.
\end{equation}
 
We remark that \(\mathcal{T}(\exp(-\beta V/2))(x)\) is the tensor train approximation of \(\exp(-\beta V/2)\), which is a \(d\)-dimensional tensor. The term \(K_d\) is defined as the Kronecker product of \(d\) matrices, which can be interpreted as a tensor train approximation to the heat kernel with rank 1.}

Combining Theorems \ref{Thm_TT_accuracy} and \ref{Thm_TT_cross}, we observe that the score function 
\[
\nabla\log(\rho_T)(x_j) = \frac{\nabla \rho_T(x_j)}{\rho_T(x_j)}
\]
can be computed efficiently and accurately when $\exp(-\beta V(x)/2)$ is sufficiently smooth and has relatively small TT ranks. Before we proceed, we would like to provide an example to demonstrate the improvement of accuracy and efficiency of approximating the integral \eqref{def_normal} for a special case where $\rho_0 = 1$ and $N(y) = 1$. In this case, the function $\nabla\log(\rho_T)$ can be explicitly written out for some cases by computing the proximal of $V(x)$ directly. We pick a challenging non-convex potential function $V(x) = |x|_{1/2}$, the computational time and accuracy using TT integration and MC integration are summarized in the following tables.
\begin{table}[H]
\parbox{.56\linewidth}{
\centering
 \begin{tabular}{l|l|lll}
& \makecell{TT\\ Integration} & \multicolumn{2}{c@{}}{\text{MC Integration with}}   \\  \cline{3-4}
&  & $10^3$ Samples & $10^6$ Samples   \\\hline
Error & 0.099     & 0.224   & 0.190 \\
Time  &$5$ & $<1$ & $85$
\end{tabular}
\caption{Relative error and time for different numbers of samples in seconds. The number of grid points in each dimension is $24\times L$ distributed in $[-L,L]$ with $L = 6$. The numerical rank of the tensor train is 3.}}
\hfill
\parbox{.4\linewidth}{
\centering
\begin{tabular}{l|ll}
$d$    &\makecell{TT\\ Integration} &\makecell{MC\\ Integration}  \\\hline
 10  & 0.27        &  31.48     \\
100 &6.29     & 277.81   \\
200& 16.53& 2186.5
\end{tabular} 
\label{table_TT_time}
\caption{Computational time in seconds among different dimensions $d$ \blue{with the same level of relative error about $0.1$.}}
.}
\end{table}

\blue{
}

 
 In summary, TT integration demonstrates substantial advantages in terms of accuracy and efficiency when compared to MC integration, especially for some high-dimensional integrations.
 
Given the improved accuracy and efficiency of using the tensor train to compute integral as in \eqref{exp_int}, we present the following theorem that rigorously demonstrates the improvement in the accuracy of the sampling algorithm. This is done for a representative scenario where both $\Pi_k$ and $\Pi^*$ are Gaussian distributions. For the notional sake, we will write $f(n,d)\lesssim g(n,d)$ when $f(n,d)\leq Cg(n,d)$ for a positive constant $C$ independent of $n$ and $d$.

\blue{\begin{theorem}
\label{Thm_compare_accu}
    For BRWP in \cite{BRWP_2023}, let $V(x) = x^{\ts}x/(2\sigma^2)$, and denote $X_k$ and $\tilde{X}_k$ be samples obtained at the $k$-th iteration with TT integration and MC iteration with means $\mu_k$ and $\tilde{\mu}_k$ respectively. Assuming all steps in the BRWP Algorithm are exact except the step of approximating $N(x)$ in \eqref{def_normal}. If the complexities of the two numerical integration methods are of the same order, and let $\mu_{\infty} = \lim_{k\rightarrow \infty}\mu_k$, $\tilde{\mu}_{\infty} = \lim_{k\rightarrow \infty}\tilde{\mu}_k$, we have
    \[
        |\mu_{\infty} - \mu^*| \lesssim \frac{d}{2^{n}}\,, \quad |\tilde{\mu}_{\infty} - \mu^*| \lesssim \frac{1}{d^{1/2} n }\,,
    \]
    where $n$ represents the number of discretization points in each dimension for TT integration. 
\end{theorem}}

\begin{proof}
    From \cite{BRWP_2023} and our analysis in the next section, under exact arithmetic, $\mu_{\infty}$ and $\tilde{\mu}_{\infty}$ all converge to $\mu^*$. Hence, we shall examine the numerical errors generated by the two different integration methods when approximating $N(y)$ defined in \eqref{def_normal}. 

    Let $N_1(y)$ and $N_2(y)$ be the approximations to $N(y)$ by TT integration and MC integration, respectively. We have
    \[
        |N_2(y) - N(y)| \lesssim n_{mc}^{-d/2}\,,
    \]
    where  $n_{mc}^d$ is the number of samples used in MC integration, involving $\mathcal{O}(n_{mc}^d)$ flops.

    For $N_1$, considering $n$ quadrature points in each dimension and truncating the integration to $[-L, L]^d$, then the error between $N_1(y)$ and $N(y)$ consists of two parts: one from truncation of integration to a bounded domain denoted as $E_1$, and the other from the quadrature rule denoted as $E_2$.
    
    For $E_1$, using the upper bound for the error function of the Gaussian distribution, we have
    \[
        E_1 \lesssim d\frac{\exp(-\beta L^2/(4\sigma^2))}{L}\,.
    \]
    For $E_2$, recalling the classical error bound of the Gauss-Legendre quadrature, we derive
    \[
        E_2 \leq d\frac{(2L )^{2n+1}(n!)^4}{((2n)!)^3}\frac{||H_{2n}||_{\infty}}{(4\beta^{-1}\sigma^2)^{2n}}\,,
    \]
    where $H_{2n}$ is the Hermitian polynomial coming from the $n$-th derivative of the Gaussian. Stirling's approximation for $n!$ and the upper bound of Hermitian polynomials yield  $E_2 \lesssim \frac{d}{2^n}$.
    Thus, $|N_1(y) - N(y)| \lesssim \frac{d}{2^n}$.

    Assuming $P$ and $\tilde{P}$ are computed with the same order of flops, we have $n_{mc} = (n^2d)^{1/d}$, and the error for MC integration becomes
    \[
        |N_2(y) - N(y)| \lesssim \frac{1}{n d^{1/2}}\,.
    \]

    Finally, recalling the iterative relation in \cite{BRWP_2023}, the numerical error in approximating the normalization term $N(x)$ will propagate to all the remaining iterations. Hence, the error in the iteration at infinity will be bounded by the same term.
\end{proof}

To conclude, according to Theorem \ref{Thm_compare_accu}, if we employ TT integration to approximate the normalization term \eqref{def_normal}, the computational accuracy will experience a significant improvement while maintaining the same order of computational complexity compared with MC integration.  

\blue{
We note that a more general theorem comparing the accuracy of TT integration and MC integration can be formulated for broader classes of distributions beyond the diagonal Gaussian distribution. In such cases, an additional error term associated with the tensor train approximation would be introduced. This extra error term is dependent on the computational complexity. To maintain clarity and focus, we present the simpler version here, which effectively illustrates the superiority of TT integration.}

The theorem \ref{Thm_compare_accu} underscores the necessity of applying TT approximation, and the introduced straightforward algorithm which employs TT integration to compute the normalization term in \eqref{def_normal} becomes particularly useful when $T$ is extremely small, and an empirical distribution for $\rho_0$ is acceptable.
However, as revealed in \cite{BRWP_2023}, the aforementioned algorithm is still biased, meaning that the steady state of the generated distribution differs from the target distribution. Therefore, in the next section, we will propose an unbiased new sampling algorithm with the assistance of TT approximation.
 
\section{Tensor Train based Noise-free Sampling Algorithm  (TT-BRWP)}
\label{sec_analysis}
In this section, we propose a new sampling algorithm with the help of a TT algorithm that utilizes a delicate choice of covariance matrix to enhance the accuracy of the BRWP proposed in \cite{BRWP_2023}. We verify the convergence and mixing time of the proposed algorithm in a representative scenario that the distribution we would like to sample is a general Gaussian distribution. 

We recall that to compute the terminal density function in BRWP, the empirical distribution is employed to estimate initial density $\rho_0$. However, this choice, along with the discretization of the particle evolution equation, leads to biased estimations of the sample mean and variance \cite{BRWP_2023}. \blue{Given that, we choose to use the following density estimation which is only implementable when the TT algorithm is employed to approximate the density function $\rho_{0,k}$ at the $k$-th interaction as}
\begin{equation}
\label{def_KDE}
\rho_{0,k}(y) = \frac{1}{M}\sum_{j=1}^M\rho_{0,k,j}(y,x_j) = \frac{1}{M}\sum_{j=1}^M\frac{\exp\bigg(-\frac{1}{2}(y-\tilde{x}_j)^{\ts} H_k^{-1}(y-\tilde{x}_j)\bigg)}{  |H_k|^{1/2} (2\pi)^{d/2}}   \,, \quad \blue{x_j \sim \Pi_k}\,,
\end{equation}
with a special choice of the covariance matrix $H_k$ and $\tilde{x}_j$ is a recalled version of $x_j$ whose definition can be found in \eqref{def_H0}. The distribution $\Pi_k$ has the density function $\rho_k$.  In subsequent analysis, we will derive and justify the explicit choice of $H_k$ and expression for $\tilde{x}_j$, which depends on the sample covariance and parameters $\beta$, $T$. \blue{From our numerical experience and the following analysis of several representative scenarios, we remark that the corrected initial density $\rho_{0,k}$ with $H_k$  will lead to a sampling algorithm with faster convergence and unbiased estimation. The more complete theoretical treatment will be addressed in future works. 
}

Before the analysis, we first present the TT-BRWP algorithm as follows.
\begin{algorithm}[H]
\caption{Tensor train BRWP sampling algorithm (TT-BRWP)}
\label{TT_Algo}
\begin{algorithmic}[1]
\item \textbf{Input}: Given $X_0 = \{x_{0,j}\}_{j=1}^M \in  \mathbb{R}^d$ $\sim \Pi_0$.
\item \textbf{Construction}: Solve backward heat equation by \eqref{exp_int} and use TT crossing to obtain TT approximation for $1/{{\eta}}_0$ in \eqref{Heat_rho}.
\For{$k = 1,2,\cdots$ Number of iterations}
\For{$j = 1,2,\cdots$ Number of samples (parallel)}
\State \textbf{Construct}: \blue{Write $\rho_{0,k,j}$ as in \eqref{def_KDE}, which is a Gaussian distribution.}
\State \textbf{Compute} $\hat{\eta}_0 =   \rho_{0,k,j}/\eta_0$ by Hadamard product.
\State \textbf{Solve} forward heat equation with $\hat{\eta}_0$ to get $\hat{\eta}_T$ and $\nabla\hat{\eta}_T$.
\State \textbf{Compute} $\rho_{T,k,j} = \eta_T\circ \hat{\eta}_T$   and $\nabla \rho_{T,k,j}  = \eta_T\circ \nabla \hat{\eta}_T$. 
\EndFor
\State \textbf{Interpolate} $\nabla \rho_T = \sum_j \nabla \rho_{T,k,j}$ and $\rho_T = \sum_j \rho_{T,k,j}$ on $X_k$  (interpolation on a gridded array).
\State \textbf{Compute} score function $\nabla \log \rho_T = \nabla \rho_T /\rho_T$. 
\State \textbf{Solve} \eqref{particle_evolution} using semi-backward Euler scheme to have $X_{k+1} = X_k - h(\nabla V(X_k)+\beta^{-1}\nabla \log(\rho_T(X_k)))$.  
\EndFor
\item \textbf{Output}:  $X_k = \{x_{k,j}\}_{j=1}^M $ for $k \in \mathbb{N}$.
\end{algorithmic}
\end{algorithm}

\blue{We remark that for each iteration, as indicated on line 5, each $\rho_{0,k,j}$ is a simple Gaussian distribution and for the case $H_k$ is diagonal, each $\rho_{0,k,j}$ is a rank 1 tensor train which implies the following three steps are very fast to compute.}

In the following, section \ref{sec_TT_BRWP_Gaussian} and \ref{sec_TT_BRWP_Bayersian} will verify Algorithm \ref{TT_Algo} by computing its mixing time and deriving continuous analog to show the convergence of the generated samples to a desired stationary distribution when the underlying distribution is assumed to be a Gaussian. In section \ref{sec_TT_numerics}, details of numerical implementations will be addressed.

\subsection{Analysis of TT-BRWP for Gaussian Distribution.}
\label{sec_TT_BRWP_Gaussian}
We first focus on the verification of Algorithm \ref{TT_Algo} when the target distribution is a multivariate Gaussian distribution. We then derive the convergence and mixing time of the proposed algorithm. \blue{In this subsection, for simplicity, we will drop the subscript $k$ in $\rho_{0,k}$ at the $k$-th iteration.}

Let us consider the case that the density function for the target distribution is
\[
\rho^*(x) = \frac{1}{Z}\exp\bigg(- \beta\frac{(x-\mu)^{\ts}\Sigma^{-1}(x-\mu)}{2}\bigg)
\]
which is a multivariate Gaussian distribution. Then the solution of the regularized Wassertain proximal defined in \eqref{regu_PDE} can be written as \cite{kernel_proximal}
\begin{equation}
\label{rho_T_gaussian}
    \rho_T(x) = \int_{\mathbb{R}^d}\rho_0(y) \frac{\exp\bigg(-\frac{\beta}{2}\bigg(\frac{||x-y||_2^2}{2T}+(x-\mu)^{\ts}\Sigma^{-1}(x-\mu)\bigg)\bigg)}{N(y)} dy \,,
    \end{equation}
    where
 \begin{equation}
    N(y) = \int_{\mathbb{R}^d} \exp\bigg(-\frac{\beta}{4}(z-\mu)^{\ts}\Sigma^{-1}(z-\mu)\bigg)\exp\bigg(-\beta\frac{||z-y||_2^2}{4 T}\bigg)dz \,.
\end{equation}

Next, we initialize the algorithm by letting $\Pi_0$ be a Gaussian distribution with mean $\mu_0$ and covariance $\Sigma_0$, which is a practically common choice for real applications. Moreover, we assume that we have sufficiently many random samples from $\Pi_0$ so that $\rho_0$ in \eqref{def_KDE} can be approximated by a continuous convolution. 
\blue{Next, we define $\tilde{x}_j$ and $H_0$ in $\rho_0$ as
\begin{equation}
\label{def_H0}
    H_0 :=  \frac{1}{2}\Sigma_0 - T^2\beta^{-2}\Sigma
_0^{-1} \,,\quad \tilde{x}_j := \frac{x_j - \mu_0}{\sqrt{2}} + \mu_0\,.
\end{equation}
Then, in this case, $\tilde{x}_j$ will be a normal distribution with mean $\mu_0$ and variance $\Sigma_0/2$. We remark for the $k$-th iteration, we simply replace $\mu_0$, $\Sigma_0$ by $\mu_k$, $\Sigma_k$.
}

Now, recall the fact that the convolution of two multivariate Gaussians with mean and covariance $(\mu_1,\Sigma_1)$, $(\mu_2,\Sigma_2)$ will still be a Gaussian distribution with mean and covariance $(\mu_1+\mu_2, \Sigma_1+\Sigma_2)$. Hence, we have
\[
N(y) \sim \exp\bigg(-\frac{\beta}{4}(y-\mu)^{\ts}(\Sigma + T)^{-1}(y-\mu)\bigg)\,, \quad \rho_0(y)\sim\exp\bigg(- \frac{1}{2}(y- {\mu}_0)(H_0+\frac{\Sigma_0}{2} )^{-1}(y-{\mu}_0)\bigg)\,,
\]
where $f(x)\sim g(x)$ denotes $f(x) = Cg(x)$ for some constants $C$ that are independent of $x$.
Substituting the above two expressions into \eqref{rho_T_gaussian}, we derive
\begin{align}
&\rho_{T}(x) \sim \exp\bigg(-\beta\frac{(x-\mu)^{\ts}\Sigma^{-1}(x-\mu)}{4}\bigg)\cdot\notag\\
&\,\,\int_{\mathbb{R}^d}\exp\bigg(-\beta\frac{\frac{||x-y||_2^2}{T}-(y-\mu)^{\ts}(\Sigma + T)^{-1}(y-\mu)}{4} -\frac{(y-\mu_0)(H_0+ \frac{\Sigma_0}{2})^{-1}(y-\mu_0) }{2}\bigg)dy\notag \\
&\qquad\,\,\sim  \exp\bigg(-\frac{1}{2}(x- \mu_{0,T})^{\ts}\Sigma_{0,T}^{-1}(x-\mu_{0,T})\bigg)
\end{align}
where  
\begin{equation}
    \label{var_rhoT}
    \mu_{0,T} = K\mu_0+\frac{\beta}{2} K^T(H_0+\frac{1}{2}\Sigma_0)^{-1}K\mu\,,\quad
\Sigma_{0,T} = 2 T\beta^{-1}K+  K^{\ts}(H_0+\frac{1}{2}\Sigma_0)K\,,
\end{equation}
and $K = (I+T\Sigma^{-1})^{-1}$. Consequently, we employ the discretization of particle evolution equation \eqref{FP_Euler} to have
\blue{\begin{equation}
     X_{1} = X_0-h\nabla V(X_0)-h\beta^{-1}\nabla\log\rho_{T} = (1-h\Sigma^{-1}+h\beta^{-1}\Sigma_{0,T})X_0 - h\beta^{-1}\Sigma_{0,T}^{-1}\mu_{0,T}\,.
\end{equation}}
Hence, it will be clear that $X_1$ and also all $X_k$ will be with Gaussian distributions. Moreover, their means and covariances can be explicitly computed which is presented in the following Lemma similar to proposition 2 in \cite{BRWP_2023}.
\begin{lemma}
\label{Lemma_Gaussian_k}
    Assuming $V(x) =(x-\mu)^{\ts}\Sigma^{-1}(x-\mu)/2$ and $X_{k}$ is a Gaussian random variable with mean $\mu_k$ and covariance $\Sigma_k$, then $X_{k+1}$ will also be a Gaussian with mean and covariance as
    \begin{equation}
        \label{mu_gaussian}
        \mu_{k+1} = (I-h\Sigma^{-1} +h\beta^{-1}\Sigma_{k,T}^{-1})\mu_{k} -h\beta^{-1}\Sigma_{k,T}^{-1}\mu_{k,T} 
    \end{equation}
    and
    \begin{equation}
    \label{var_gaussian}
        \Sigma_{k+1} = (I-h\Sigma^{-1} + h\beta ^{-1}\Sigma_{k,T}^{-1}) \Sigma_{k}(I-h\Sigma^{-1}+h\beta^{-1}\Sigma_{k,T}^{-1}) 
    \end{equation}
    where $\mu_{k,T}$ and $\Sigma_{k,T}$ are given in \eqref{var_rhoT} by replacing \blue{$\Sigma_0$ with $\Sigma_k$ and $H_0$ with $H_k$}.
\end{lemma}

Now, we are ready to justify our choice of the covariance matrix $H_{k}$ in \eqref{def_H0} at the $k$-th iteration, which will be
\begin{equation}
\label{def_bandwidth}
    H_k
    = \frac{1}{2}\Sigma_k - \beta^{-2}T^2\Sigma
_k^{-1}  \quad \Rightarrow \quad H_k + \frac{\Sigma_k}{2} = \Sigma_k - \beta^{-2}T^2\Sigma_k^{-1}\,,
\end{equation}
where $\Sigma_k$ is the covariance matrix for particles $X_k$ in the $k$-th iteration. Then by \eqref{var_gaussian}, the steady state of covariance for $\rho_k$ at $k\rightarrow \infty$ denoted as $\Sigma_{\infty}$ will satisfy
\begin{equation}
\label{steady_state}
I + h(\beta^{-1} \Sigma_{\infty,T}^{-1}-\Sigma^{-1}) = I \Rightarrow  \Sigma-T^2\Sigma^{-1} = \beta\Sigma_{\infty}-T^2\beta^{-1}\Sigma_{\infty}^{-1}\,,
\end{equation}
where we have substituted the expression for $\Sigma_{\infty,T}$ by letting $k\rightarrow \infty$ in \eqref{var_rhoT}. We note $\Sigma_{\infty} = \beta^{-1}\Sigma$ is a solution to the above relationship \eqref{steady_state}. Moreover, under the assumption that $\Sigma$ is a positive definite and commutes with $\Sigma_{\infty}$, $\Sigma_{\infty} = \beta^{-1}\Sigma$ will also be the unique positive definite solution to \eqref{steady_state}. In other words, the steady state of $\Pi_k$ will have the same covariance as the target distribution which is $\beta^{-1}\Sigma$.  To allow $H_k$ to be positive definite,  $T$ shall be sufficiently small whose choice will be presented rigorously in the following analysis.

To better understand the convergence of $\Sigma_k$ and $\mu_k$, we first consider the simplified case where $\Sigma_0$ and $\Sigma$ commute, indicating they share the same eigenspace. Moreover, we assume $\mu= 0$ as the general case can be deduced from translation. Then from \eqref{var_gaussian}, we observe that $\Sigma_k$ also commutes with $\Sigma$ for all $k$ in this case.

We introduce the notations $\sigma^{(j)}$, $\sigma^{(j)}_{k,T}$ and $\sigma_k^{(j)}$ as the $j$-th eigenvalue of $\Sigma$, $\Sigma_{k,T}$, and $\Sigma_k$ respectively. Moreover, we write $m_k^{(j)}$ as the $j$-th entry in $\mu_k$. In the following, we will skip the superscript if the result holds for all $j$.
By simplifying expressions \eqref{var_rhoT} 
we have
\begin{equation}
    \sigma_{k,T}^2 = \frac{2T\beta^{-1}}{1+T/\sigma^{2}} + \frac{ \sigma_k^2-T^2\beta^{-2}/\sigma_k^2}{(1+T/\sigma^2)^2}
\end{equation}
and then by Lemma \ref{Lemma_Gaussian_k}, the evolution of $m_k$ and $\sigma_k^2$ can be computed as follows
\begin{align}
\label{mu_iteration}
    m_{k+1} = &\bigg(1-\sigma^{-2}h + \frac{h \beta^{-1} \sigma^{-2}T(1+\sigma^{-2}T)}{\sigma_k^2- T^2\beta^{-2}\sigma_k^{-2} + 2T\beta^{-1}(1+\sigma^{-2}T)}\bigg)m_k\,,\\
\label{sigma_iteraion}
    \sigma_{k+1}^2 = &\bigg(1-\sigma^{-2}h + \frac{h \beta^{-1}(1+\sigma^{-2}T)^2}{\sigma_k^2-T^2\beta^{-2}\sigma_k^{-2}+2T\beta^{-1}(1+\sigma^{-2}T)}\bigg)^2 \sigma_{k}^2\,,
\end{align}
and the equation for the steady state $\sigma_{\infty}$ in \eqref{steady_state} simplifies to
\begin{equation}
    \sigma^{2} - T^2\sigma^{-2} = \beta\sigma_{\infty}^2 - T^2\beta^{-1}\sigma_{\infty}^{-2}  
\end{equation}
with the only positive solution being $\sigma_{\infty}^2 = \beta^{-1}\sigma^2$, indicating the only steady state of $\Sigma_k$ is exactly $\Sigma$.

To characterize the behavior of the evolution of sample distribution $\Pi_k$ we get at the $k$-th iteration, we introduce the total variation $d_{TV}(P,Q)$ between two distributions $P$ and $Q$ over $\mathbb{R}^d$ with probability density functions $p$ and $q$:
\begin{equation}
    d_{TV}(P,Q) = \int_{\mathbb{R}^d}|p(x)-q(x)|dx\,.
\end{equation}
We also recall the result about the total variation between two Gaussian distributions with the same mean from \cite{TV_Gaussian}
\begin{equation}
\label{TV_Gau}
    d_{TV}(\mathcal{N}(\mu,\Sigma_1),\mathcal{N}(\mu,\Sigma_2)) = \frac{3}{2}\min\left\{1,\sqrt{\sum_{i=1}^d\lambda_i^2}\right\}\,,
\end{equation}
where $\lambda_i$ are eigenvalues of $\Sigma_1^{-1}\Sigma_2 - I$.

Next, we introduce the mixing time concerning the target distribution $\Pi^*$ as
\begin{equation}
    t_{\text{mix}}(\varepsilon,\Pi_0) = \min_k \{k | d_{TV}(\Pi_k, \Pi^*) \leq \varepsilon\}\,.
\end{equation}
 
Now, we present a theorem regarding the mixing time of the proposed TT-BRWP, assuming the initial and exact means are equal. The result is comparable with Theorem 2 in \cite{BRWP_2023} with slightly different assumptions and a simplified proof. The evolution of sample means $m_k$ will be discussed in Section \ref{sec_TT_BRWP_Bayersian}.

Before the main theorem, we prove the following technical lemma, which will be critical in the proof of the next theorem.
\begin{lemma}
\label{Tech_lemma}
Let $T,h\in [0,\sigma^2/4]$, $\beta^{1/2} \zeta\in[\sigma/2 , 3\sigma/2]$, and $\beta\geq 1$, the function 
\[
    f(\zeta) := 1-h\left[\sigma^{-2} +\beta^{-1}(1+ \sigma^{-2} T)^2 \frac{3T^2\beta^{-2}/\zeta^2+\zeta^2-2T\beta^{-1}(1+\sigma^{-2} T )}{(\zeta^2-T^2\beta^{-2}/\zeta^2+2T\beta^{-1}(1+ \sigma^{-2}T))^2}\right]
\]
satisfies $|f(\zeta)|\leq 1-h/\sigma^2$.
\end{lemma}
\begin{proof}
   Firstly, to show $f(\zeta)\leq  1-h/\sigma^2$, we observe that it suffices to show
\[
3T^2\beta^{-2}/\zeta^2+\zeta^2-2T\beta^{-1}(1+\sigma^{-2} T ) > 0
\]
for $\zeta$ such that $|\beta^{1/2}\zeta-\sigma|\leq \sigma/2$. This is equivalent to exploring the minimum value of a quadratic polynomial in $\zeta^2$, and the desired condition will boil down to
\[
T^2(3-(1+T/\sigma^2)^2) \geq T^2
\]
which is true by our choice of $T$. 

Secondly, we show the given condition on $T$ and $h$ ensures that $f(\zeta)>0$. Rearranging terms, $f(\zeta)>0$ is equivalent to
\begin{equation}
    \frac{\beta}{(1+\sigma^{-2}T)^2}\bigg(\frac{1}{h}-\frac{1}{\sigma^2}\bigg) \geq  \frac{3T^2\beta^{-2}/\zeta^2+\zeta^2-2T\beta^{-1}(1+\sigma^{-2} T )}{(\zeta^2-T^2\beta^{-2}/\zeta^2+2T\beta^{-1}(1+ \sigma^{-2}T))^2}\,.
\end{equation}
We note that since $\beta^{1/2}\zeta\in [\sigma/2,3\sigma/2]$, an upper bound for the right-hand side can be obtained as
\[
\frac{3T^2\beta^{-2}/\zeta^2+\zeta^2-2T\beta^{-1}(1+\sigma^{-2} T )}{(\zeta^2-T^2\beta^{-2}/\zeta^2+2T\beta^{-1}(1+ \sigma^{-2}T))^2} \leq \frac{12T^2/(\beta\sigma^2)+9\sigma^2/(4\beta)}{(\sigma^2/(4\beta)- 4T^2/(\beta\sigma^2))^2}\,,
\]
where we used the fact that $\beta\zeta^2\geq T^2$. Then since $\beta\geq 1$, we now note it suffices to show
\begin{equation}
      \frac{1}{(1+\sigma^{-2}T)^2}\bigg(\frac{1}{h}-\frac{1}{\sigma^2}\bigg) \geq \frac{12T^2/(\sigma^2)+9\sigma^2/4}{(\sigma^2/4- 4T^2/\sigma^2)^2}\,.
\end{equation}
Then we will observe the inequality holds after substituting the choice of $T,h\leq \sigma^2/4$.

\end{proof}

\begin{theorem}
\label{Thm_gaussian}
Let $V(x) = x^{\ts}\Sigma^{-1}x/2$ and $\sigma_m$ be the smallest eigenvalue of ${\Sigma}$. For distribution $\Pi_k$ evolved following Algorithm \ref{TT_Algo} with $T,h\in [0,\sigma_m^2/4]$, assuming $\Sigma_0$ and $\Sigma$ commute and $|\beta^{1/2}\sigma^{(j)}_0-\sigma^{(j)}|\leq \sigma^{(j)}/2$ for all $j$, we have
\[
||\Sigma_k-\beta^{-1}\Sigma||_F \leq (1-\delta)^k||\Sigma_0-\beta^{-1}\Sigma||_F\,,
\]
where $\delta = h/\sigma_m^2$.
Moreover, if we assume $\mu_0 = 0$, the total variation is bounded by
\[
d_{\text{TV}}(\Pi_k,\Pi^*)   \leq \frac{3}{2}C\sqrt{d}(1-\delta)^k\,,
\]
and the mixing time is given as
\[
t_{mix}(\varepsilon,\Pi_0)  = \mathcal{O}\left(\log (d/\epsilon)/\log(1-h/\sigma_m^2) \right) 
\]
where the constant $C$ depends on $\Sigma$, $\Sigma_0$, and $T$.
\end{theorem}
\begin{proof}
Under the assumption that $\Sigma_0$ and $\Sigma$ commute and share the same eigenspaces, it is sufficient to focus on the evolution of the $j$-th eigenvalue of $\Sigma_k$ for each iteration which is given by \eqref{sigma_iteraion} as
\begin{equation}
\sigma_{k+1}  = \bigg(1-h \sigma^{-2} + h\frac{ \beta^{-1}(1+\sigma^{-2} T)^2}{\sigma_{k}^2-T^2\beta^{-2}\sigma_{k}^{-2}+2T\beta^{-1}(1+\sigma^{-2}T)}\bigg)\sigma_k\,,
\end{equation}
where we skip the superscript $j$ to simplify notations.
We regard the evolution of $\sigma_k$ as a fixed-point iteration with the iterative function
\begin{equation}
\label{iterative_fun}
\phi(\zeta)  = \bigg(1-h\sigma^{-2} + h\frac{\beta^{-1}(1+\sigma^{-2} T)^2}{\zeta^2-T^2\beta^{-2}/\zeta^2+2T\beta^{-1}(1+\sigma^{-2}T)}\bigg)\zeta\,.
\end{equation}
The derivative of $\phi$ with respect to $\zeta$ is
\begin{equation}
\phi'(\zeta) = 1-h\bigg[\sigma^{-2} +\beta^{-1}(1+ \sigma^{-2} T)^2 \frac{3T^2\beta^{-2}/\zeta^2+\zeta^2-2T\beta^{-1}(1+\sigma^{-2} T )}{(\zeta^2-T^2\beta^{-2}/\zeta^2+2T\beta^{-1}(1+ \sigma^{-2}T))^2}\bigg] \,.
\end{equation}
By Lemma \ref{Tech_lemma}, $|\phi'(\zeta)|\leq 1-\delta$, showing that $\sigma_k$ converges linearly with rate $1-\delta$. Hence, we show the convergence of $\Sigma_k$.

Leveraging the TV norm between two Gaussians in \eqref{TV_Gau}, we derive
\begin{equation}
d_{\text{TV}}(\Pi_k, \Pi^*) \leq \frac{3}{2}C\sqrt{d}(1-\delta)^k\,,\quad C = \min\{1,\max_j\{(\sigma^{(j)})^{-1}(\sigma^{(j)}-\sigma^{(j)}_0)\}\}\,.
\end{equation}
Furthermore, the mixing time can be derived by combining the above estimate on the TV norm and its definition.  
\end{proof}

We remark that as $\beta$ becomes larger, the initial guess should be more accurate to converge to the correct variance. 

Before we proceed to the more general case, we would like to compare the convergence of variance with the initial density defined in \eqref{def_KDE} depicted in the above theorem with the case of using the empirical distribution as $\rho_0$ proposed in \cite{BRWP_2023}. The result in the next figure comes from the iterative function we formulated in \eqref{iterative_fun}.
\begin{figure}[H]
    \centering
    \includegraphics[scale = 0.33]{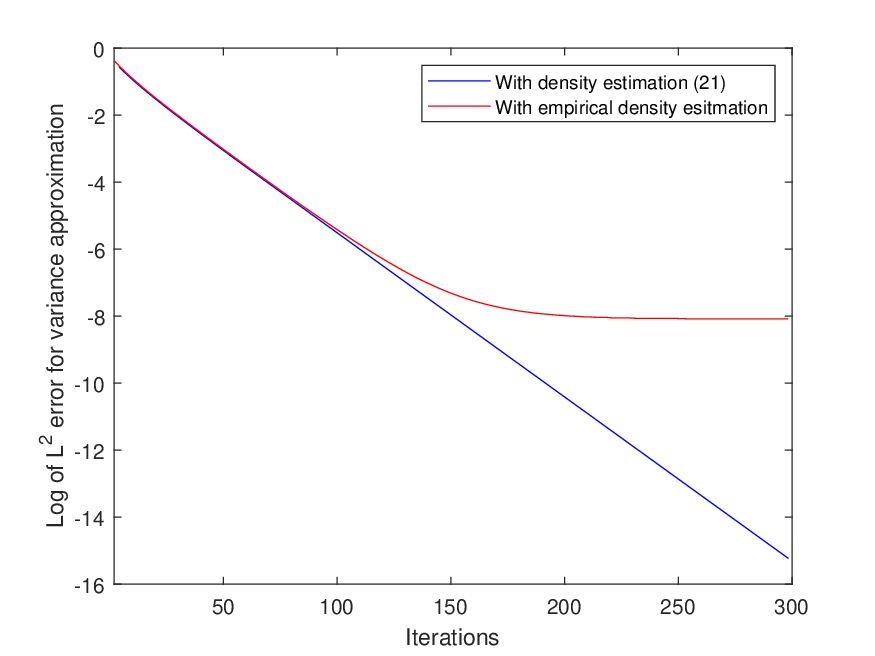}
    \includegraphics[scale = 0.33]{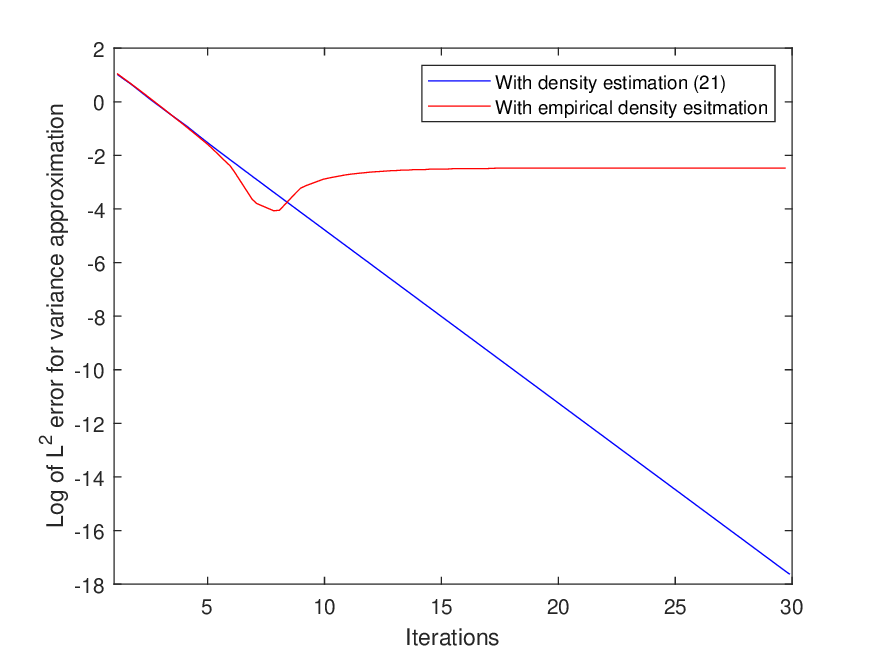}
        \includegraphics[scale = 0.33]{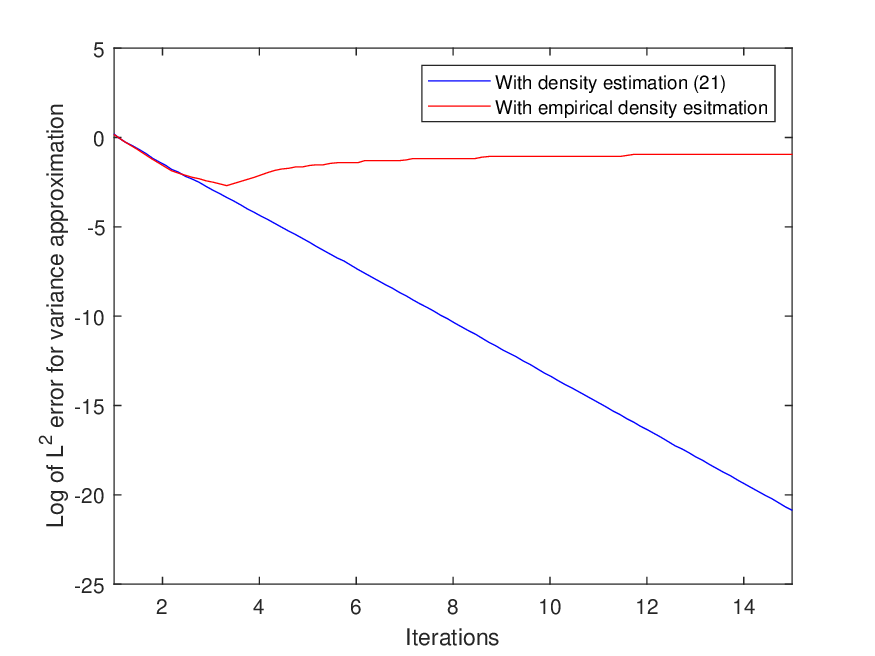}
    \caption{\blue{Logarithm of the approximation error for variance versus iteration using empirical distribution (red) and density estimation defined in \eqref{def_KDE} (blue) with $\sigma = 2,0.5,0.25$ (from left to right), $T = 0.1$, $h=0.1$. }}
    \label{fig_var_error}
\end{figure}
From Fig.\,\ref{fig_var_error}, we can observe that the unbiased nature of Algorithm \ref{TT_Algo} by employing a delicate choice of the covariance matrix for $\rho_0$ which improves the accuracy of variance approximation, especially for the case where $\sigma$ is relatively small. Moreover, we also hope to numerically explore the influence of parameter $T$ on the rate of convergence which is shown in the following figure with the help of iterative function \eqref{iterative_fun}.
\begin{figure}[H]
    \centering
    \includegraphics[scale = 0.5]{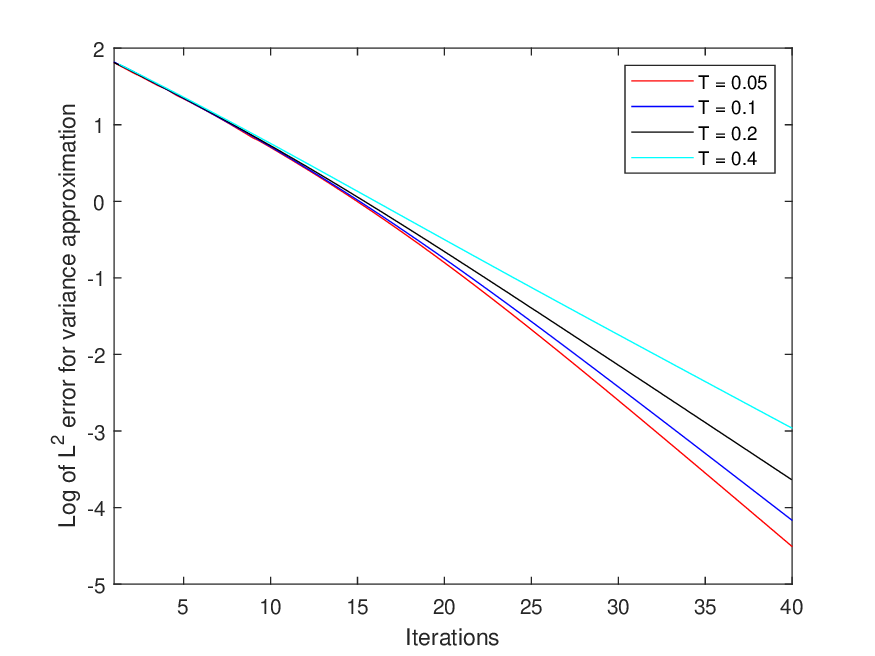}
    \includegraphics[scale = 0.5]{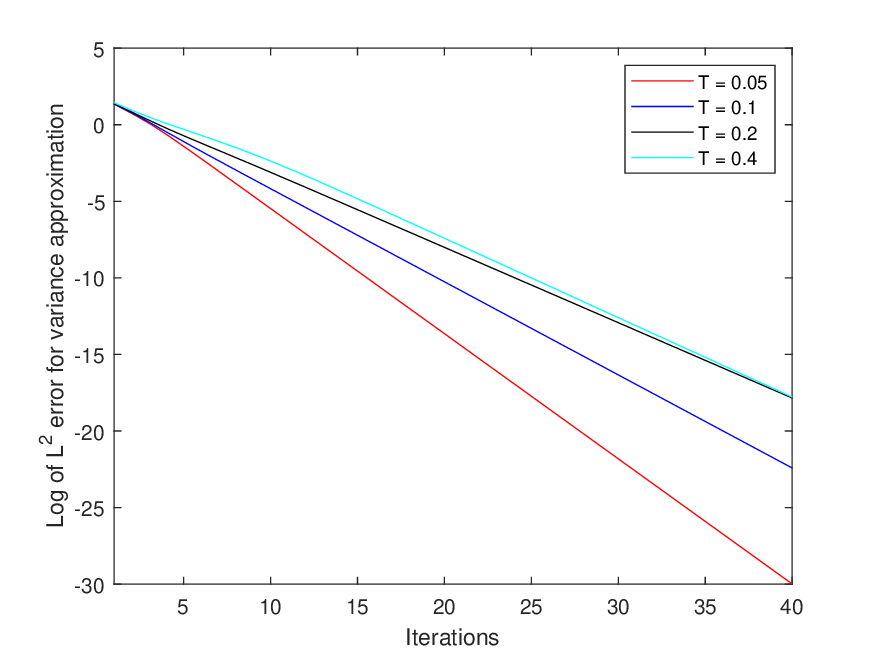}
     \caption{Logarithm of the approximation error of Algorithm \ref{TT_Algo} for variance versus iterations using different terminal time $T$ with $T= 0.05,0.1,0.2,0.4$, $h=0.1$, $\sigma = 0.5$ (Left) and $1$ (Right). }
    \label{fig_T_error}
\end{figure}
From Fig.\,\ref{fig_T_error}, we observe that as $T$ becomes smaller, a slightly faster convergence can be achieved. However, as $T$ becomes smaller, the integration in \eqref{def_normal} requires more nodal points to obtain a satisfactory numerical approximation. Hence, we will choose a fixed $T$ empirically in our following numerical experiments. 

Next, we void the assumption that $\Sigma_0$ and $\Sigma$ commute with each other and derive a continuous analog of the evolution of $\Sigma_k$ to further demonstrate its convergence property. First, from the results in Theorem \ref{Thm_gaussian} and numerical experiments in Fig.\,\ref{fig_T_error}, the algorithm is interesting when $T$ and $h$ are both small. Hence, it is reasonable to drop high-order terms in $h$ and $T$.

Let us revisit the iterative relationship about $\Sigma_k$ in \eqref{var_gaussian} and notice that the evolution of the covariance matrix can be written as
\begin{align}
\label{eqn_approximate_Sigma}
    \Sigma_{k+1} = &\Sigma_k - h(\Sigma^{-1}\Sigma_k+\Sigma_k\Sigma^{-1} - \beta^{-1}\Sigma_{k,T}^{-1}\Sigma_k-\beta^{-1}\Sigma_k\Sigma_{k,T}^{-1}) + \mathcal{O}(h^2)\,.
\end{align}
Then recalling $H_k+\Sigma_k/2 = \Sigma_k -  \beta^{-2}T^2\Sigma_k^{-1}$ and by the Neumann series on \eqref{var_rhoT}, we notice that
\begin{align*}
      &\Sigma_{k,T}^{-1} =   (2T\beta^{-1}+\Sigma_k-T\Sigma^{-1}\Sigma_k - T\Sigma_k\Sigma^{-1} + \mathcal{O}(T^2))^{-1} \\
      = &  \Sigma_k^{-1}(I + T(\Sigma^{-1}+\Sigma_k\Sigma^{-1}\Sigma_k^{-1}-2\beta^{-1}\Sigma_k^{-1}))+ \mathcal{O}(T^2)\,.
\end{align*}
Substituting this into \eqref{eqn_approximate_Sigma}, we arrive at
\begin{align}
    \Sigma_{k+1} 
    =& \Sigma_k + \mathcal{O}(h^2)+\mathcal{O}(T^2) \\
    -
    &h\big[\Sigma_k\Sigma^{-1} + \Sigma^{-1}\Sigma_k-2\beta^{-1} - T\beta^{-1}(2\Sigma^{-1}+\Sigma_k^{-1}\Sigma^{-1}\Sigma_k +\Sigma_k \Sigma^{-1}\Sigma_k^{-1}- 4\beta^{-1}\Sigma_k^{-1})\big]\notag \,.
\end{align}
Now, let $h\rightarrow 0$ and consider the continuous analog of the covariance matrix as $\Sigma_t$ depends on time $t$. This leads to
\begin{equation}
\label{cont_sigma_de}
    \frac{d\Sigma_t}{d t} = - \Sigma_t\Sigma^{-1}-\Sigma^{-1}\Sigma_t + 2\beta^{-1} + T\beta^{-1}(2\Sigma^{-1}+\Sigma_t^{-1}\Sigma^{-1}\Sigma_t +\Sigma_t \Sigma^{-1}\Sigma_t^{-1}- 4\beta^{-1}\Sigma_t^{-1})+\mathcal{O}(T^2)\,.
\end{equation}
To further explore the convergence of $\Sigma_t$, we consider the Frobenius norm of $(\Sigma_t-\beta^{-1}\Sigma)^{\ts}(\Sigma_t-\beta^{-1}\Sigma)$, which is equivalent to the trace of the matrix. By noting the factorization 
\begin{align*}
&2\Sigma^{-1}+\Sigma_t^{-1}\Sigma^{-1}\Sigma_t +\Sigma_t \Sigma^{-1}\Sigma_t^{-1}- 4\beta^{-1}\Sigma_t^{-1} \\
= &\Sigma_t^{-1}\Sigma^{-1}(\Sigma_t-\beta^{-1}\Sigma) + (\Sigma_t-\beta^{-1}\Sigma)\Sigma^{-1}\Sigma_t^{-1}+2\Sigma^{-1}(\Sigma_t-\beta^{-1}\Sigma)\Sigma_t^{-1}\,,
\end{align*}
we can derive
\begin{align}
    &\frac{d||(\Sigma_t-\beta^{-1}\Sigma)^2||_F}{d t} = 
  \Tr\{2(\Sigma_t-\beta^{-1}\Sigma)^T\frac{d\Sigma_t}{dt}\}
    \\=&\Tr\{(\Sigma_t-\beta^{-1}\Sigma)(-4\Sigma^{-1}+T\beta^{-1}\Sigma^{-1}\Sigma_t^{-1}+T\beta^{-1}\Sigma_t^{-1}\Sigma^{-1})(\Sigma_t-\beta^{-1}\Sigma) \\
    &+  2T\beta^{-1} (\Sigma_t-\beta^{-1}\Sigma) \Sigma^{-1} (\Sigma_t-\beta^{-1}\Sigma)\Sigma_t^{-1}\}+\mathcal{O}(T^2)\notag \\
    \leq& -4\Tr\{(\Sigma_t-\beta^{-1}\Sigma)\Sigma^{-1}  (\Sigma_t-\beta^{-1}\Sigma)\} (I-T\beta^{-1}\rho\{ \Sigma_t^{-1}\})+\mathcal{O}(T^2)\,,
\end{align}
where we have used the facts that $\Tr\{AB\} = \Tr\{BA\}$ and $|\Tr\{AB\}|\leq |\Tr\{A\}|\rho\{B\}$ where $\rho$ is the spectral radius of $B$.
The above implies that when $(I-T\beta^{-1}\rho(\Sigma_t^{-1}))$ is positive, which is guaranteed for small $T$, $||(\Sigma_t-\beta^{-1}\Sigma)^2||_F$ decays to $0$. The conclusion of the above discussion is summarized in the following theorem.

\begin{theorem}
\label{Thm_Frobinus}
    When $h$ and $T$ are sufficiently small, let  $\Sigma_t$ be the continuous analog of $\Sigma_k$ when $h\rightarrow 0$ and assume that $\Sigma_t$ is symmetric positive definite for all $t$. Then the Frobenius norm of $\Sigma_t-\Sigma$ will decay to $0$ for $T$ small enough, and the decay rate is bounded by
\begin{align}
  \frac{d||(\Sigma_t-\beta^{-1}\Sigma)^2||_F}{d t}  \leq -4\Tr\{(\Sigma_t-\beta^{-1}\Sigma)\Sigma^{-1}  (\Sigma_t-\beta^{-1}\Sigma)\} (I-T\beta^{-1}\rho\{ \Sigma_t^{-1}\})+\mathcal{O}(T^2) \leq 0
\end{align}
which depends on eigenvalues of $\Sigma_t$.
\end{theorem}
The above theorem shows, from the continuous perspective, that the covariance matrix of $X_k$ will converge to $\beta^{-1}\Sigma$ when $T$, $h$ that are sufficiently small. 

To conclude this section, we provide a brief comparison between our theoretical convergence result and BRWP as well as MALA.
\begin{itemize}
    \item In comparison to BRWP, the proposed sampling algorithm is unbiased, leading to a steady state that exactly matches the target distribution while BRWP introduces a variance shift by a factor of $T^2/(\beta \sigma_m^2)$, as illustrated in Fig., \ref{fig_var_error}.
 
    \item  Contrasting with MCMC-type methods like MALA, which has a theoretical upper bound for mixing time of $\mathcal{O}(d^2)$ \cite{MALA_2019_converge}, our method exhibits a faster convergence with a theoretical bound of $\mathcal{O}(\log(\sqrt{d}))$. This suggests that our approach achieves faster convergence in higher dimensions, a characteristic that will be evident in numerical experiments.
\end{itemize}




 
\subsection{Analysis of TT-BRWP Algorithm for a Simplified Bayesian Inverse Problem}
\label{sec_TT_BRWP_Bayersian}

In this section, we shall focus on the accuracy of the proposed TT-BRWP in an interesting scenario that arises in Bayesian inverse problems and data fitting. The potential function will be
\begin{equation}
\label{V_Bay}
    V(x) = \frac{||A x - \mu||_2^2}{2\lambda^2} + ||\Gamma x||_2^2
\end{equation}
where $A$ and $\Gamma$ are known linear forward operators and linear regularization operators respectively that could be $m\times d$ for $m\neq d$, $\mu$ is the noisy observation, and $\lambda$ is the noise level that has been estimated. This model corresponds to Tikhonov regularization for inverse problems and $L^2$ regularization in data fitting. For $V(x)$, we can rewrite it as
\begin{equation}
    V(x) = \bigg(x-\bigg(\frac{A^{\ts}A}{2\lambda^2}+\Gamma^{\ts}\Gamma\bigg)^{-1}\frac{A^{\ts}}{2\lambda^2}\mu\bigg)^{\ts}\bigg(\frac{A^{\ts}A}{2\lambda^2}+\Gamma^{\ts}\Gamma\bigg)\bigg(x-\bigg(\frac{A^{\ts}A}{2\lambda^2}+\Gamma^{\ts}\Gamma\bigg)^{-1}\frac{A^{\ts}}{2\lambda^2}\mu\bigg) + C
\end{equation}
where $C$ is independent of $x$. Let $\tilde{\mu} = (A^{\ts}A/(2\lambda^2)+\Gamma^{\ts}\Gamma)^{-1}A^{\ts}/(2\lambda^2)\mu$ and $\tilde{\Sigma} = (A^{\ts}A/(2\lambda^2)+\Gamma^{\ts}\Gamma))^{-1}$, then our goal is to draw samples from the distribution $\mathcal{N}(\tilde{\Sigma},\tilde{\mu})$ with the proposed TT-BRWP to estimate $\tilde{\mu}$.

We note that since $A$ and $\Gamma$ are available, $\tilde{\Sigma}$ and $\Sigma_k$ will commute.
Now, as in the previous section, we still assume $\tilde{\mu}=0$ since the general case can be obtained by a change of coordinates. For the evolution of sample mean $\mu_k$, we focus on one entry $m_k = \mu_k(j)$ and $\tilde{m} = \tilde{\mu}(j)$ for some index $j$ which by \eqref{mu_iteration} satisfies
\begin{equation}
\label{simplifed_mu}
     m_{k+1} =  \bigg(1- h\sigma^{-2}\frac{  T(1+\sigma^{-2}T)+(\sigma_k^2-T^2\sigma_k^{-2})}{\sigma_k^2- T^2\sigma_k^{-2} + 2T(1+\sigma^{-2}T)}\bigg)m_k\,,\\
\end{equation}
when $\beta = 1$.

To compute the total variation between two Gaussian distributions with different mean values, we recall the following result which is a simplified version of Theorem 1.8 in \cite{arbas2023polynomial}.
\begin{lemma}
\label{lemma_TV_gau_diff}
    Suppose $\Sigma_1$, $\Sigma_2$ commute and $d_{TV}(\mathcal{N}(\mu_1,\Sigma_1 ),\mathcal{N}(\mu_2,\Sigma_2)) \leq 1/600$, we have
    \[
    d_{\text{TV}}(\mathcal{N}(\mu_1,\Sigma_1 ),\mathcal{N}(\mu_2,\Sigma_2)) \leq \frac{1}{\sqrt{2}}\max\bigg\{||\Sigma_1^{-1}||_F||\Sigma_1-\Sigma_2||_F ,||\Sigma^{-1/2}( \mu_1-\mu_2)||_2 \bigg\}\,.
    \] 
\end{lemma}
Then we can derive the following result regarding the total variation difference and mixing time. 
\begin{theorem}
\label{thm_Bayersian}
    Let $V(x)$ be defined in \eqref{V_Bay} and $\sigma_m$ be the smallest eigenvalue of $\tilde{\Sigma}$. For distribution $\Pi_k$ evolved following Algorithm \ref{TT_Algo} with $T\leq \sigma_m^2/(2+2\sqrt{2})$, $h\leq \sigma_m^2/4$, assuming the initial distribution satisfies  $d_{TV}(\Pi_0,\Pi^*)\leq 1/600$ and $\sigma^{(j)}_0=\sigma^{(j)}$ for all $j$, we have
\[
||\mu_k-\tilde{\mu}||_{\infty} \leq (1-\delta)^k||\mu_0-\tilde{\mu}||_{\infty}\,,
\]
where $\delta = h/(2\sigma_m^2)$.
Moreover, the total variation is bounded by
\[
d_{\text{TV}}(\Pi_k,\Pi^*) \ \leq \frac{1}{\sqrt{2}}C\sqrt{d}(1-\delta)^k\,,
\]
and the mixing time is given as
\[
t_{mix}(\varepsilon,\Pi_0)  = \mathcal{O}\left(\log (d/\varepsilon)/\log(1-h/(2\sigma_m^2)) \right), 
 \]
where the constant $C$ depends on $\mu$, $\mu_0$, and $T$.
\end{theorem}
\begin{proof} To show $||\mu_k-\tilde{\mu}||_{\infty}\leq(1-\delta)^k ||\mu_0-\tilde{\mu}||_{\infty}$, by the iterative relation in \eqref{simplifed_mu}, since we have $\sigma_k = \sigma$, it can be simplified as
\[
m_{k+1} = \bigg(1-h\frac{T+\sigma^2}{\sigma^4+2T\sigma^2-T^2}\bigg)m_k\,.
\]

We observe it suffices to verify
\begin{equation}
    h\leq (T+\sigma^2)- \frac{2T^2}{(T+\sigma^2)} \leq \frac{h}{\delta}\,,
\end{equation}
where we note the condition on $T$ ensures that $(\sigma^4+2T\sigma^2-T^2)>0$.
Then we may take $f(t) :=  (t+\sigma^2)- {2t^2}/{(t+\sigma^2)}$ and it suffices to show $f(0)$, $f(\sigma^2/4)$ satisfy the above inequalities and $f(t)$ is monotonic.

Next, all the above conditions will be equivalent to
\[
    h \leq \sigma^2 \leq h/\delta\,, \quad
    h \leq 23/20\sigma^2 \leq h/\delta\,,\quad 
    \sigma^4 - 2T\sigma^2-T^2 \geq 0\,;
\]
which are satisfied by our choice of parameters. 

Finally, the statement on total variation norm and mixing time follows from Lemma \ref{lemma_TV_gau_diff} and their definitions directly. 
\end{proof}

In conclusion, Theorem \ref{thm_Bayersian} establishes the linear convergence of the proposed method for the potential function \eqref{V_Bay}, which is relevant in the context of Bayesian inverse problems and data fitting applications.

Moreover, from \eqref{simplifed_mu}, we can also derive the continuous analogue $m_t$ for any component of $\mu$ which is 
\[
|m_t-\tilde{m}| = \exp\bigg(-\frac{T+\sigma^2}{(T+\sigma^2)^2-2T^2}t\bigg)|m_0-\tilde{m}|\,,  
\]
and if $(T+\sigma^2)^2-2T^2>0$, i.e., $T\leq \sigma^2/(\sqrt{2}-1)$, $m_t$ converges monotonically to  $\tilde{m}$. This provides a clear continuous analog of results in Theorem \ref{thm_Bayersian}

\subsection{Numerical Consideration and Computational Complexity}
\label{sec_TT_numerics}
In this subsection, we shall briefly discuss some numerical details in the implementation of TT-BRWP in Algorithm \ref{TT_Algo} and the efficiency of computation which provides important practical guidance on its application.

For the computation of $\hat{\eta}_T$ and $\nabla\hat{\eta}_T$ in line 7, we utilize a similar kernel formulation introduced in \eqref{exp_int} which provides a solution for the forward heat equation in \eqref{forward_heat}. Hence, we have
\begin{equation}
    \hat{\eta}_T(x) = \bigg(\frac{\beta}{2\pi T}\bigg)^{d/2}\int_{\mathbb{R}^d} \exp\bigg(-\frac{\beta||x-y||_2^2}{4T}\bigg) \hat{\eta}_0(y)dy
\end{equation}
and its gradient will be
\begin{equation}
   \nabla \hat{\eta}_T(x) = - \bigg(\frac{\beta}{2\pi T}\bigg)^{d/2}\int_{\mathbb{R}^d}\frac{\beta(x-y)}{2T} \exp\bigg(-\frac{\beta||x-y||_2^2}{4T}\bigg) \hat{\eta}_0(y)dy\,.
\end{equation}
We observe that all the above computations as well as $\hat{\eta}_T\circ{\eta}_T$ can be implemented in a tensor-train format efficiently by using $K_d$ defined in \eqref{def_Kd}.

Moreover, in algorithm \ref{TT_Algo}, when $x$ is near the boundary of the discretization grid, the value of $\eta_0$ 
\begin{equation}
    \eta_0(x)= \int_{\mathbb{R}^d}\exp\bigg(-\frac{\beta}{2}\bigg({V(x)}+\frac{||x-y||^2}{2T}\bigg)\bigg) dy \,,
\end{equation}
will be extremely small which creates underflow and induces difficulties in approximating $1/\eta_0$. In this scenario, one may compute $-\log(\eta_0)$ on line 2 and compute the exponential of a tensor by Taylor's polynomial of exponential function which enhances the numerical stability with the price of increased computational complexity. 

Finally, for the interpolation on line 10, linear interpolation on a gridded mesh is employed which has proven to be able to provide satisfactory numerical accuracy and efficiency.    
  
\section{Numerical Experiments}
\label{sec_NE}

In this section, we test the proposed TT-BRWP in algorithm \ref{TT_Algo} and compare it with the classical Unadjusted Langevin Algorithm (ULA), the unbiased Metropolis-adjusted Langevin algorithm (MALA), and the BRWP with MC integration in \cite{BRWP_2023}. For all our following experiments, \blue{the initialization $X_0$ is always sampled from a Gaussian distribution with mean $0$ and variance $1$ except for example 3. Moreover, for all the tensor train approximation, we run the crossing algorithm \cite{tt_cross} on the mesh $[-L, L]^d$ for $L = 6$ and $32L$ points in each dimension.}

    
\subsection{Gaussian Distribution}

\noindent\textbf{Example 1:} In this example, we explore the case \(V(x) = -(x-a)^{\ts}\Sigma^{-1}(x-a)/2\) in \(\mathbb{R}^6\) where \(\Sigma = \sigma^2 I_6\) for different choices of $\sigma$ to validate our theoretical results on the convergence of variance presented in Theorem \ref{Thm_gaussian}. 

\begin{figure}[H]
    \centering
        \begin{tabular}{cc}
    \includegraphics[scale=0.49]{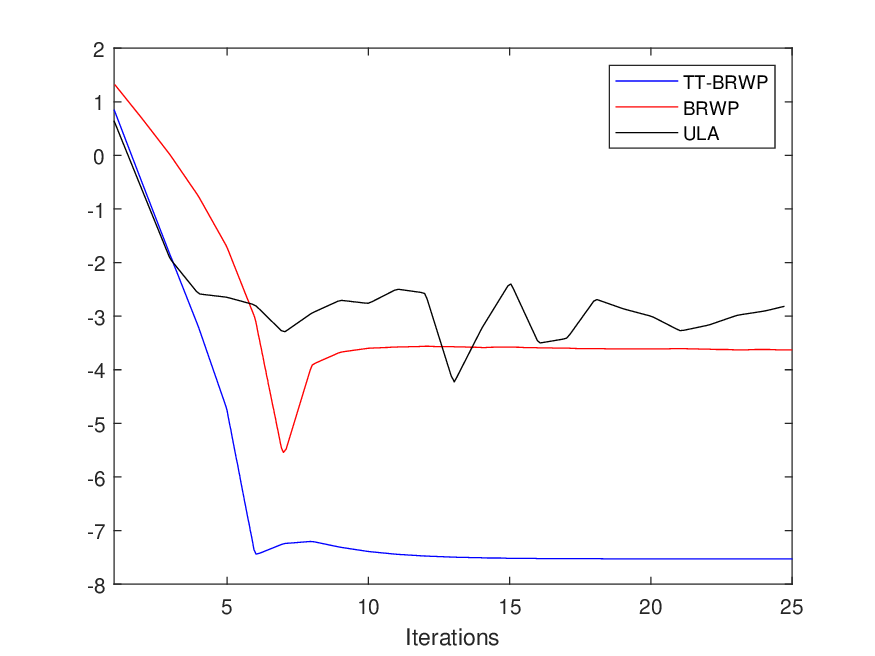}&
    \includegraphics[scale=0.49]{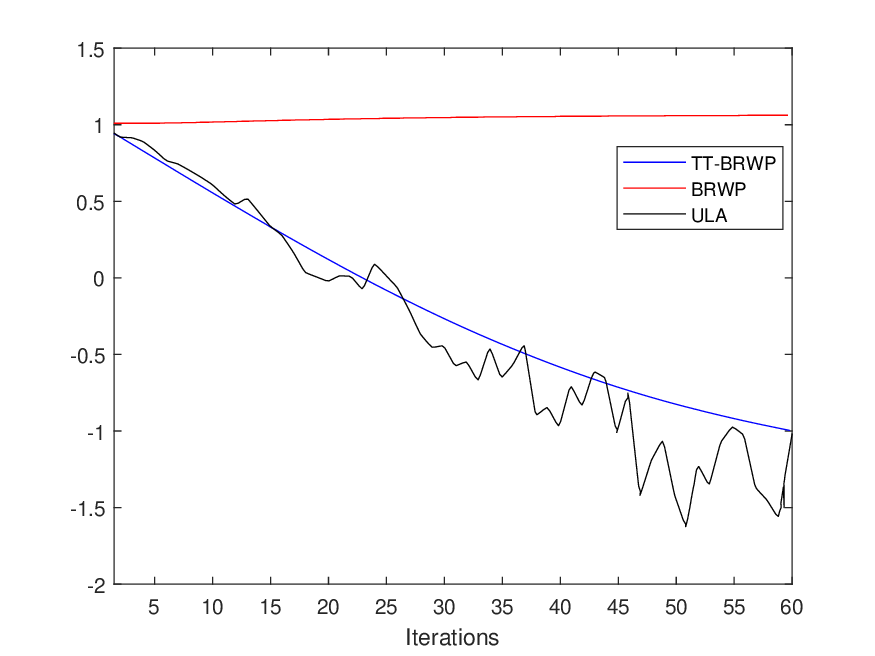}\\
 True variance $\sigma = 0.2$ & True variance $\sigma = 4$
    \end{tabular}
    \caption{Example 1. The logarithm of the $L^2$ error for the variance with samples generated from TT-BRWP (blue), BRWP (red), and ULA (black).}
    \label{fig:Gaussian_variance}
\end{figure}

From Fig.\,\ref{fig:Gaussian_variance}, we observe clear linear convergence of the variance for Gaussian distributions with both \(\sigma = 0.2\) and \(\sigma = 4\). From left, TT-BRWP exhibits faster convergence for small \(\sigma\) and improves the final estimation significantly due to the unbiased nature; from right, TT-BRWP is more robust compared to BRWP for large \(\sigma\) as BRWP degenerates in this case.

\noindent\textbf{Example 2:} Next, we consider the case when \(\Sigma\) is a  general SPD matrix. We pick
\[
\hat{\Sigma} = \begin{bmatrix}
    0.4 & 0.2 & 0.3  \\
    0.2 & 3 & 0.2  \\
    0.3 & 0.2 & 6  \\
\end{bmatrix}
\]
and let \(\Sigma = \begin{bmatrix}
    \hat{\Sigma} & \bf{0}\\
    \bf{0} & I_3
\end{bmatrix}\).
\blue{The rank of resulting TT approximation to $\exp(-\beta V/2)$ is 5.} Then we examine the evolution of the mean and the total variation for the distribution generated by TT-BRWP, BRWP, and ULA, respectively.

 \begin{figure}
    \centering
        \begin{tabular}{cc}
    \includegraphics[scale=0.49]{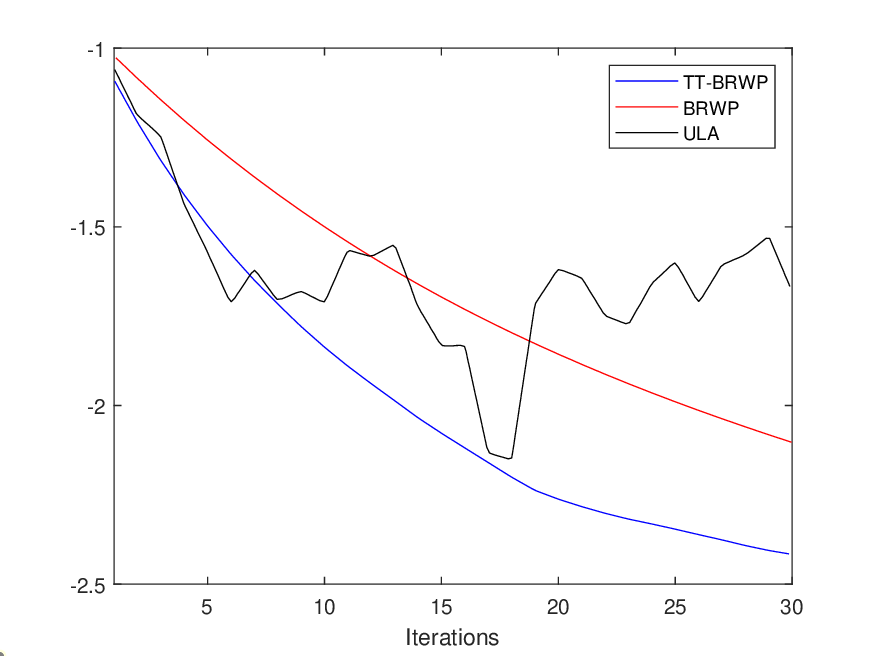}&
    \includegraphics[scale=0.49]{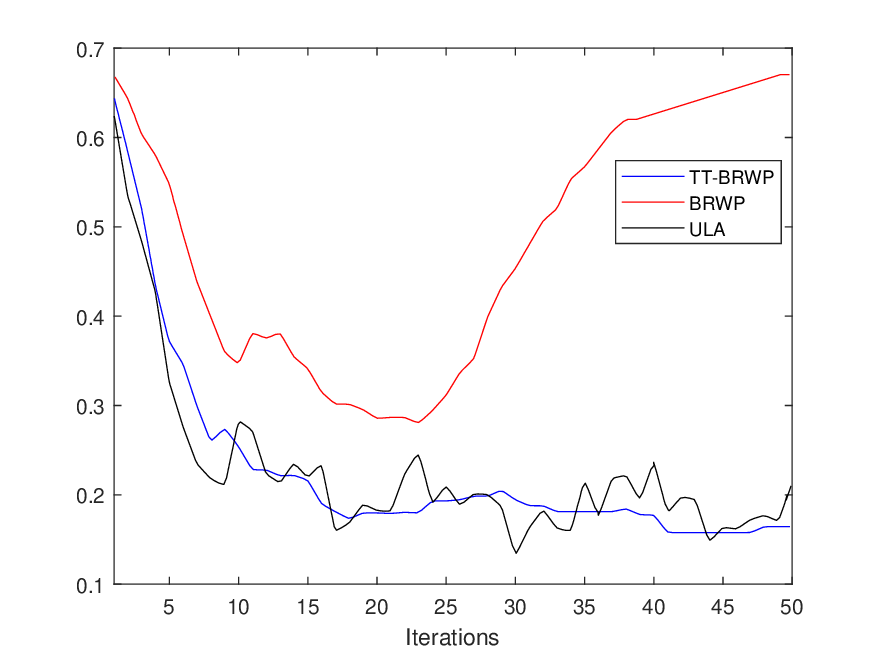}\\
    Log of $L^2$ error for mean value  & Total variation error
    \end{tabular}
    \caption{Example 2. The error of mean for samples generated for TT-BRWP (blue), BRWP (red), and ULA (black) versus iterations.}
    \label{fig:Gaussian_mean}
\end{figure}

From the first plot of Fig.\,\ref{fig:Gaussian_mean}, TT-BRWP provides the best approximation of the sample mean among the three algorithms. Moreover, from the second plot, BRWP with MC integration degenerates for this case which shows the necessity of the TT approximation. 

\subsection{Gaussian Mixture and Bimodal Distribution}
\blue{\noindent\textbf{Example 3:} In the third example, we consider sampling from a Gaussian mixture distribution in \(\mathbb{R}^4\) defined by
\[
\rho^*(x) = \frac{1}{Z} \left( \frac{\exp\left(-||x - a||_2^2 / 2\right) + \exp\left(-||x + a||_2^2 / 2\right)}{2} \right),
\]
where \(Z\) is the normalization constant. We select \(a = (2, 0, 0, 0)^{\top}\) in \(\mathbb{R}^4\), and choose parameters \(h = 0.1\) and \(T = 0.1\) to ensure the convergence of the sampling algorithms. The rank of the resulting TT approximation to \(\exp(-\beta V / 2)\) is 3. Initial distributions are drawn from a normal distribution with variance 2 and mean vectors \(0\) and \(a\) for the experiments shown in the first and second rows of Fig. \ref{fig:Gaussian_mixture}.

\begin{figure} 
    \centering
    \begin{tabular}{ccc}
        \includegraphics[scale=0.32]{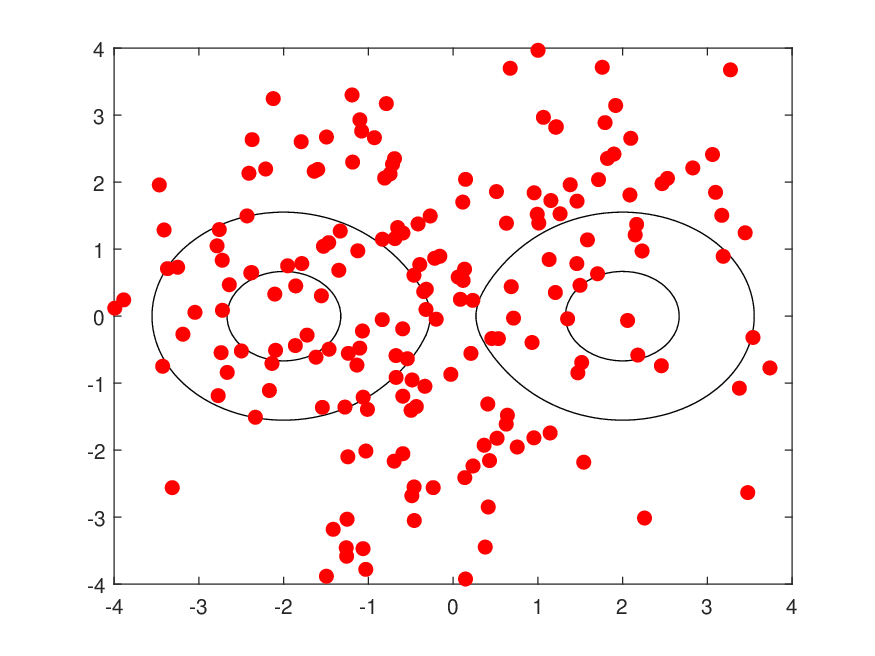} &
        \includegraphics[scale=0.32]{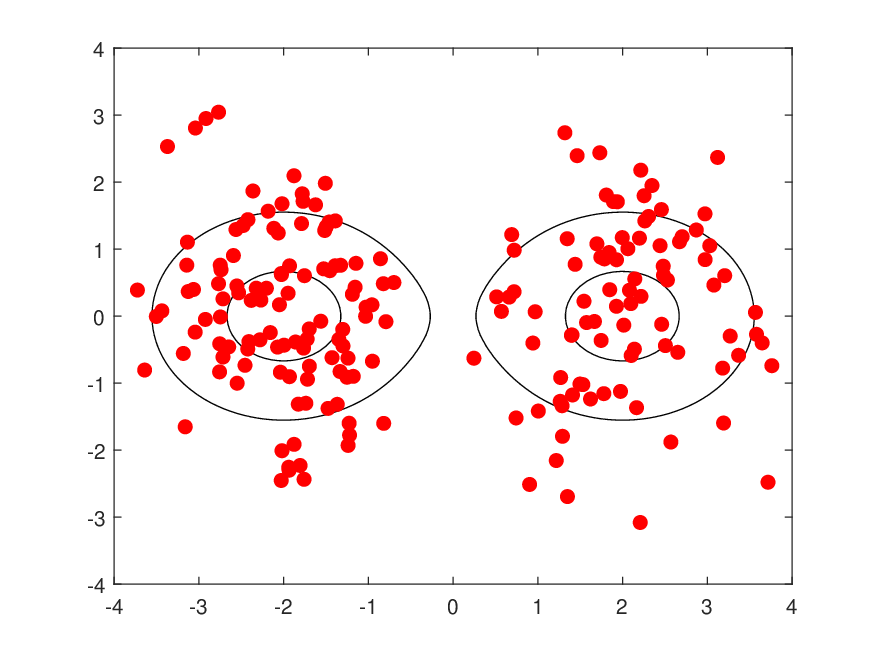} &
        \includegraphics[scale=0.32]{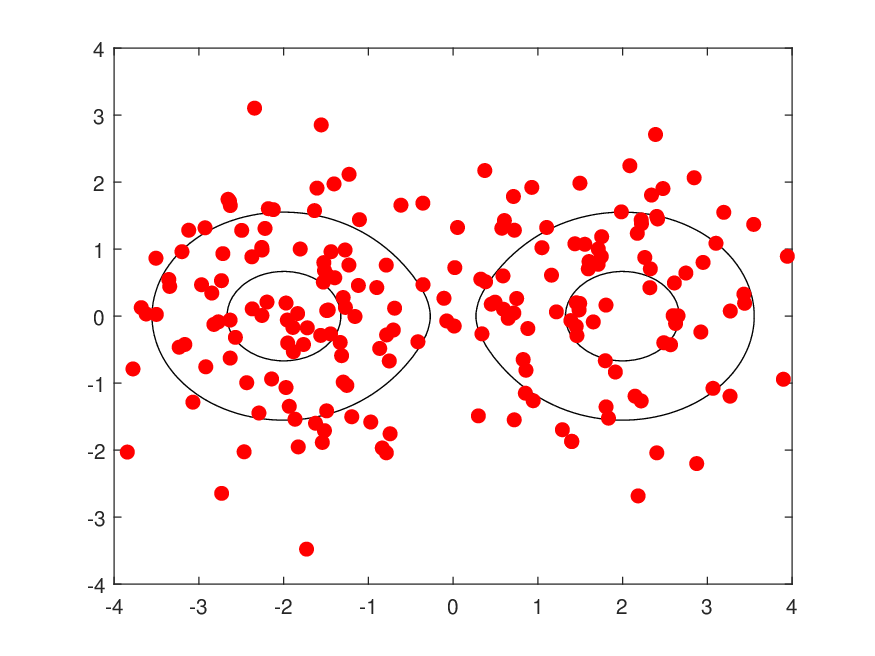}\\
      Initial Particles & TT-BRWP & ULA \\
    \end{tabular}\\
    \begin{tabular}{ccc}
        \includegraphics[scale=0.32]{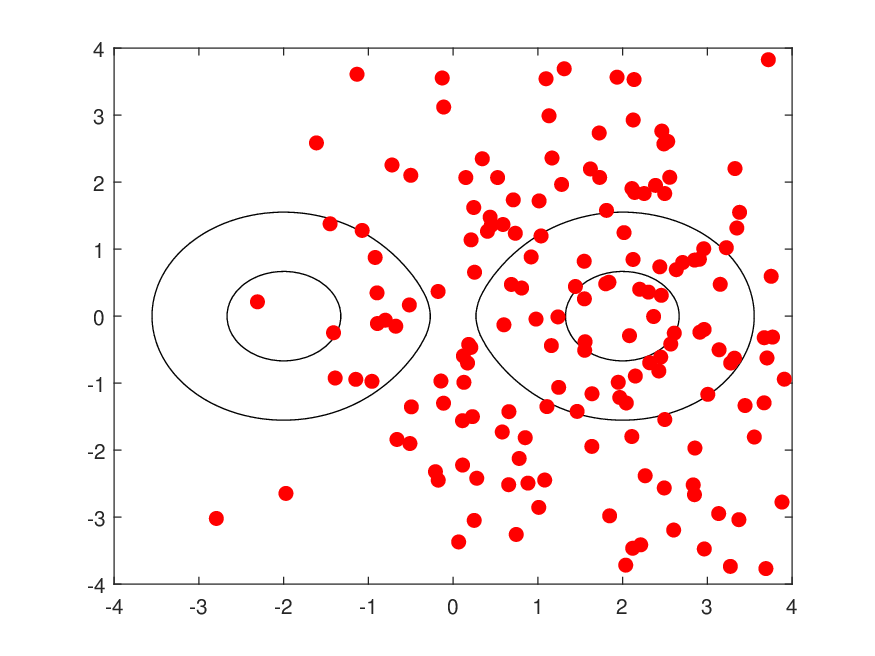} &
        \includegraphics[scale=0.32]{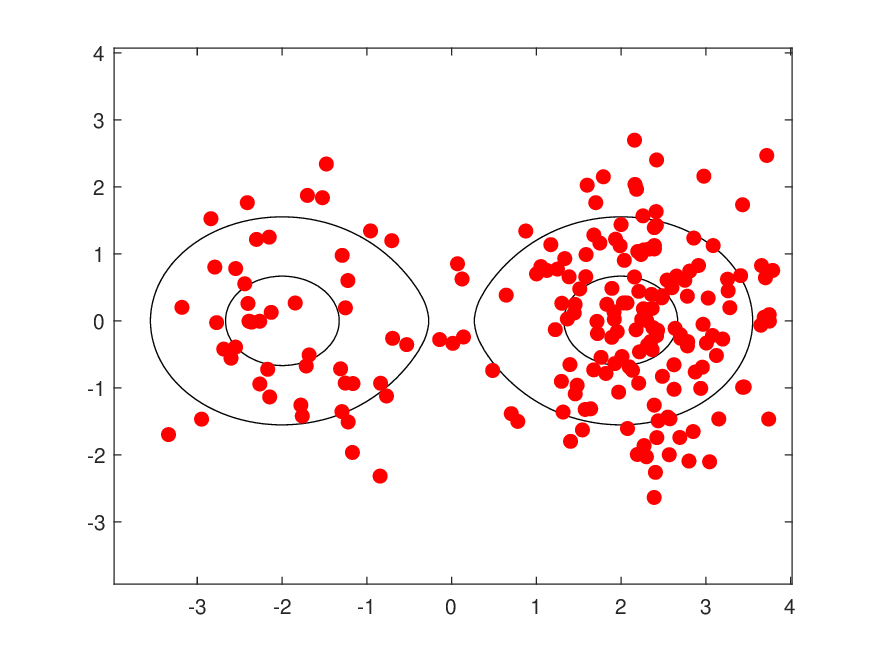} &
        \includegraphics[scale=0.32]{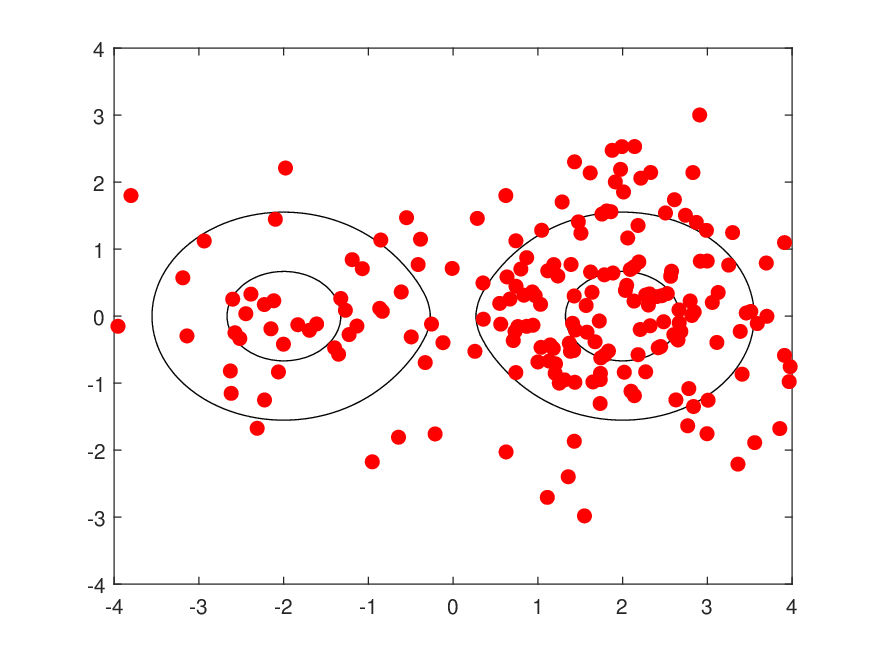}\\
      Initial Particles & TT-BRWP & ULA \\
    \end{tabular}
    \caption{\blue{Example 3. Evolution of particles for different algorithms after 10 iterations (first row) and 15 iterations (second row) under different initial distributions. The contour lines are \(0.8\) and \(0.3\) of density function.}}
    \label{fig:Gaussian_mixture}
\end{figure}

From the first row of Fig. \ref{fig:Gaussian_mixture}, we observe that under a reasonable initial distribution that covers the two modes, TT-BRWP (middle plot) provides a more structured distribution of samples compared to ULA, which has many points falling outside the high-probability region of the original distribution. In the second row of Fig. \ref{fig:Gaussian_mixture}, where the initial distribution is centered around one of the modes, TT-BRWP (middle plot) also demonstrates faster convergence compared to ULA. This indicates that the proposed approach, as an interacting particle system, converges quickly, even under poor initial distribution.
}


\noindent\textbf{Example 4:} In this example, we consider sampling from a bimodal distribution (double moon) in \(\mathbb{R}^6\) with
\[
\rho^*(x) = \frac{1}{Z}\exp(-2(||x||-a)^2)\left[\exp\left(-2(x_1-a)^2\right)+\exp\left(-2(x_1+a)^2\right)\right]\,.
\]
\blue{The rank of resulting TT approximation to $\exp(-\beta V/2)$ is $15$.}
We test the case \(a = 2\) with \(T = 0.01\) and \(h = 0.01\).

\begin{figure}[H]
    \centering
    \begin{tabular}{ccc}
        \includegraphics[scale=0.32]{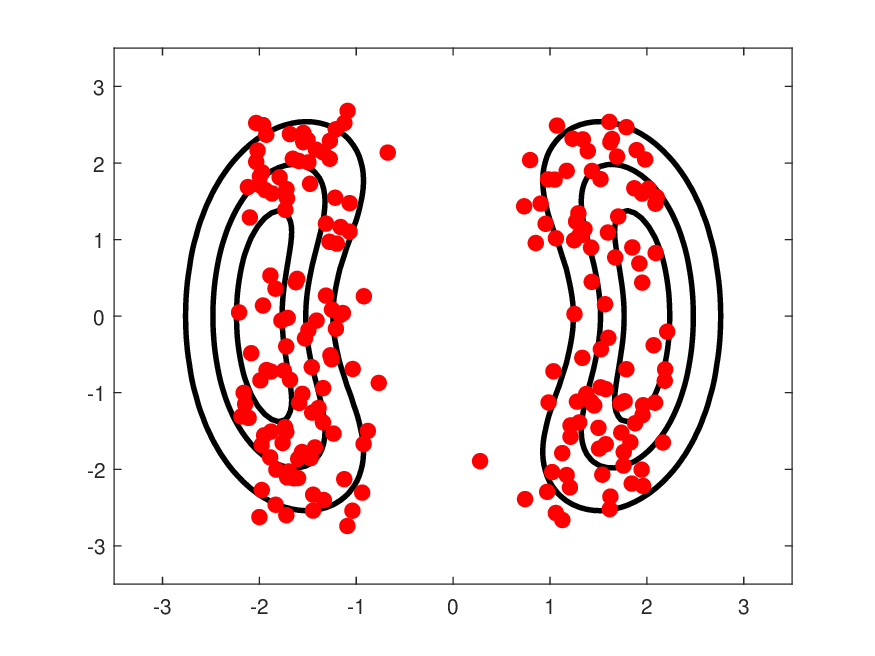} &
        \includegraphics[scale=0.32]{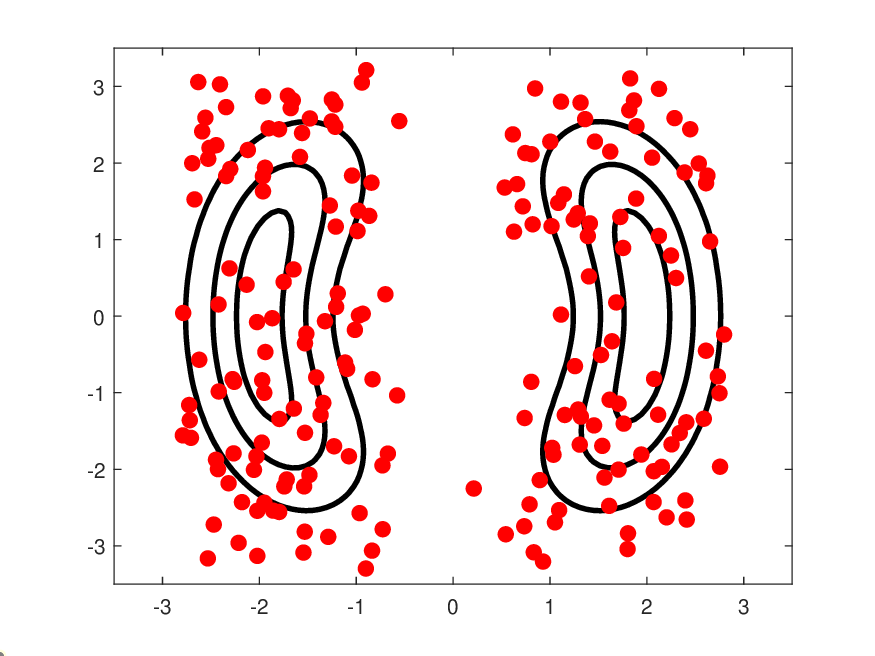} &
        \includegraphics[scale=0.32]{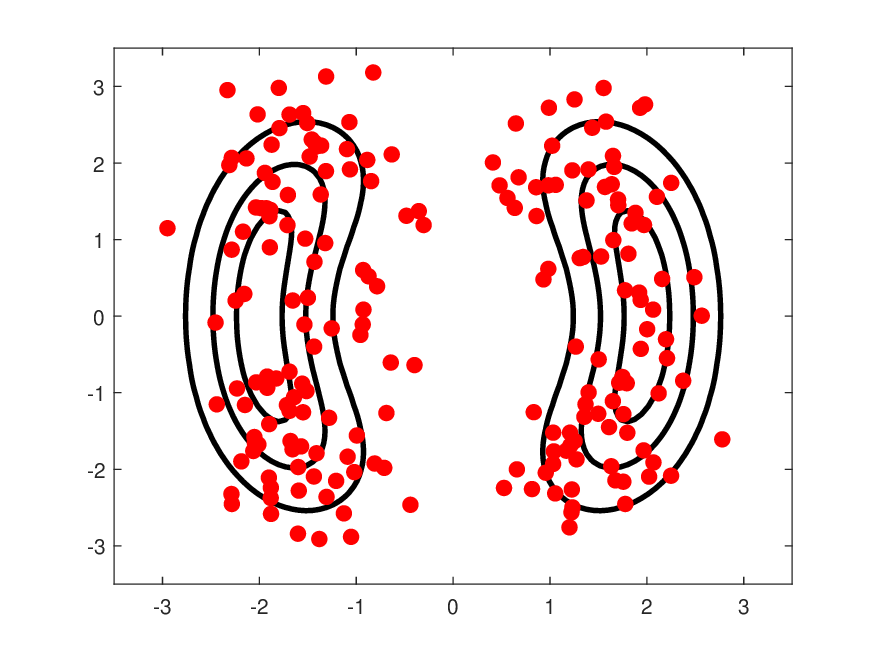} \\
          TT-BRWP & BRWP & ULA 
    \end{tabular}
    \caption{Example 4. Evolution of sample points for different algorithms after 30 iterations where the contour lines are \(0.8\), \(0.4\), and \(0.1\) in \(20\) iterations.}
    \label{fig:Bimodal_1}
\end{figure}

From Fig.\,\ref{fig:Bimodal_1}, it is evident that TT-BRWP provides a set of particles that are concentrated more in the high probability region of the target distribution within only 20 iterations, while BRWP and ULA converge more slowly. 

\subsection{Nonconvex Potential Function}

\noindent\textbf{Example 5:} In this example, we present a particularly interesting case where the potential function \(V(x)\) is nonconvex, which is known to be highly challenging due to the slow convergence of MC integration as shown in table \ref{table_TT_time}. We consider
\[
V(x) = ||x-a||_{1/2}^2 = \big(\sum_{j=1}^3 |x_j-a(j)|^{1/2}\big)^{2},
\]
with \(a = (1,1,0)^{\ts}\) in \(\mathbb{R}^3\). \blue{The rank of resulting TT approximation to $\exp(-\beta V/2)$ is 6.}

\begin{figure}[H]
    \centering
    \begin{tabular}{ccc}
        \includegraphics[scale=0.32]{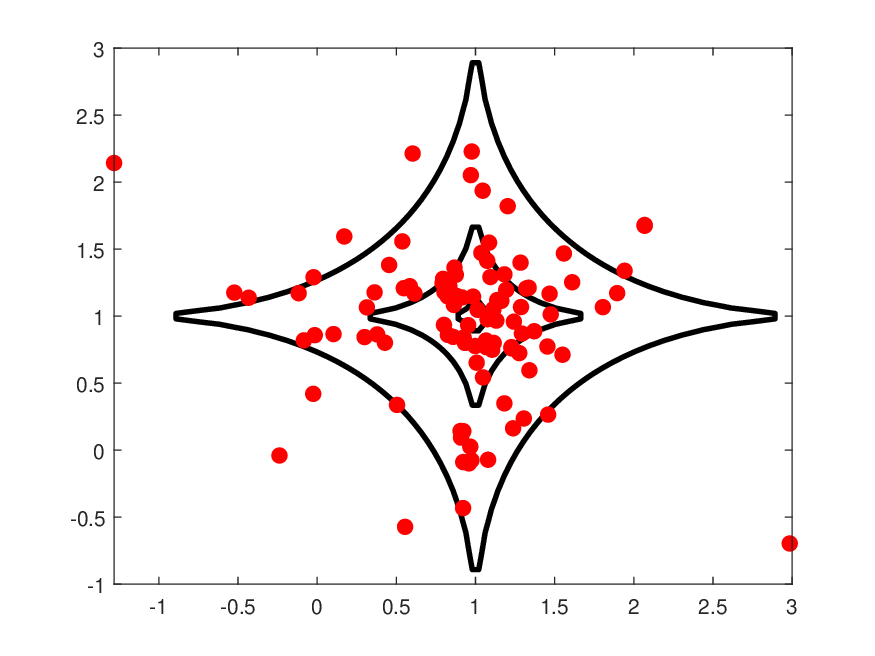} &
        \includegraphics[scale=0.32]{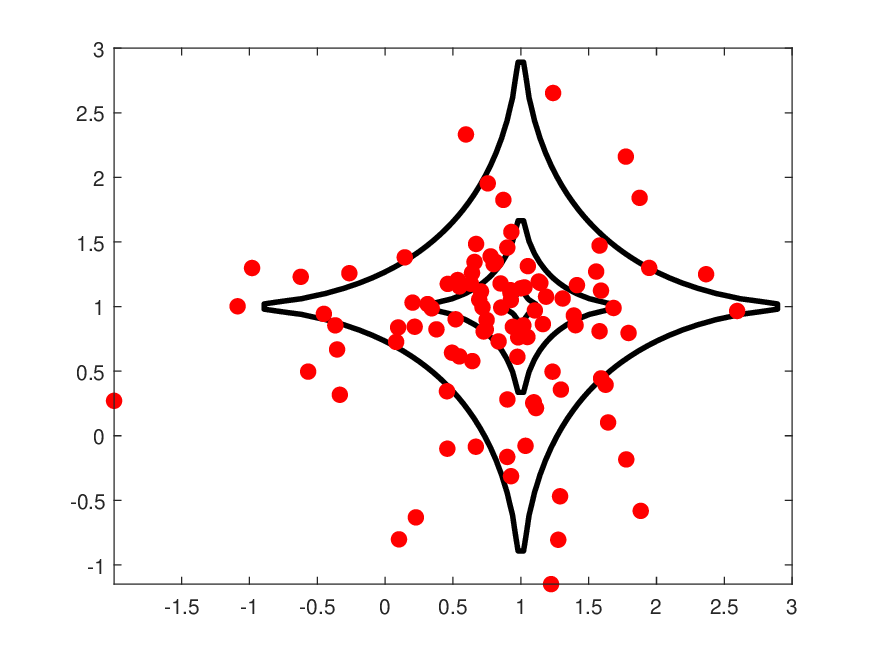} &
        \includegraphics[scale=0.32]{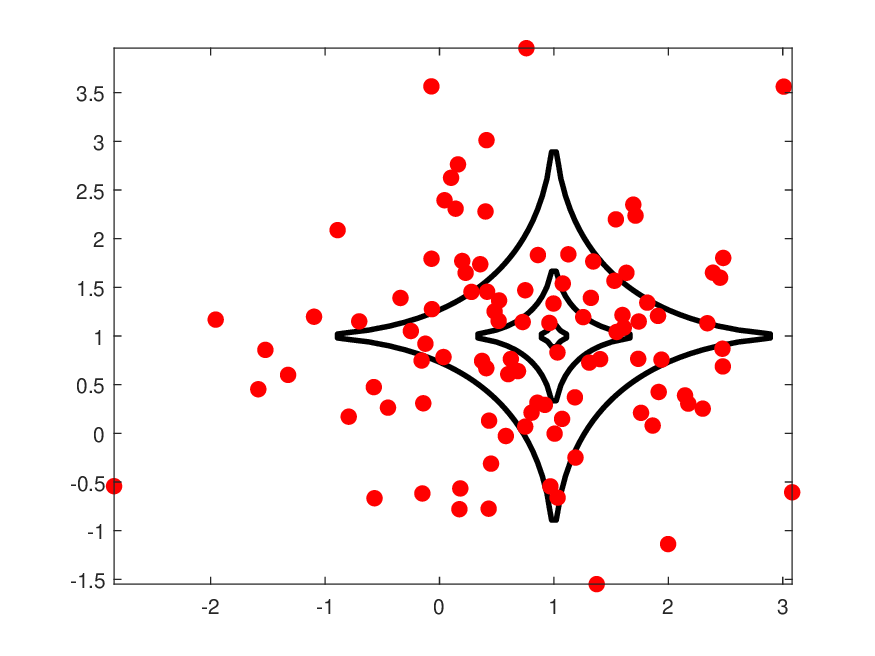} \\
        TT-BRWP& BRWP & ULA
    \end{tabular}
    \caption{Example 5. Evolution of sample points for different algorithms where the contour lines are \(0.8\), \(0.4\), and \(0.1\) after \(20\) iterations.}
    \label{fig:Nonconvex}
\end{figure}

From Fig.\,\ref{fig:Nonconvex}, we can observe that the proposed TT-BRWP algorithm provides a quite accurate set of samples that distributed follows the desired non-convex density function, while samples from BRWP with MC integration and ULA  does not exhibit clear similarity with the target density function.

\subsection{Bayesian Inverse Problems}

In this subsection, we will examine the accuracy and robustness of the proposed TT-BRWP to tackle several interesting and ill-posed inverse problems using Bayesian inference. 

Let us first recall the general setting for Bayesian inverse problems. Firstly, we write the measurement process as
\[
y = \mathcal{G}(x^*) + \zeta,
\]
where \(\zeta\) is random noise, \(x^*\) is the underlying truth we would like to recover, and \(\mathcal{G}\) is the forward operator. The objective of Bayesian inverse problems is to estimate the posterior distribution
\begin{equation}
    \pi(x|y) = \pi(y|x) \pi(x)\,,
\end{equation}
where \(\pi(x)\) is the prior density, and \(\pi(y|x)\) is the likelihood function depending on the forward operator and noise level. Existing sampling algorithms will often suffer from slow convergence or even divergence, especially for high-dimensional or nonlinear scenarios.  

\noindent\textbf{Example 7:}
The first example we consider is a classical ill-posed inverse problem for recovering the initial distribution of the heat equation. Let \(u\) be the solution of the heat equation in the sense that
\begin{align*}
\frac{\partial }{\partial t}u(t,x) = \frac{\partial^2}{\partial x^2} u(t,x) \,, \quad &u(t,0) = u(t,\pi) = 0\,, \quad u(0,x) = u_0(x)\,,\\
&u(T,x) = u_T(x)\,, \quad x\in [0,\pi]\,.
\end{align*}
Then our goal is to recover \(u_0\) from noisy measurement data \(u_T\). To simplify the computation, we assume \(u_0\) is composed of a series of trigonometric functions, and our task is to recover \(\theta_k\) in \(u_0 = \sum_{k=1}^d \theta_k \sin(kx)\) with $d = 10$.

In our experiment, let \(\theta\) represent vectors containing coefficients \(\theta_k \), \(y\) be the measurement data polluted by noise with a noise level \(\sigma^2 = 0.1\), the likelihood function is chosen as
\begin{equation}
    \pi(\theta|y) = \exp\bigg(-\frac{||u(T,x,\theta)-y||_2^2}{2\sigma^2}\bigg)\,.
\end{equation}
The prior distribution is chosen as the \(L_1\) norm of $\theta$ as a sparse constraint, and hence, the potential function \(V(\theta)\) will be
\[
V(\theta) = \frac{||u(T,x,\theta)-y||_2^2}{2\sigma^2} + \tau ||\theta||_{L^1}\,,
\]
where \(\tau\) is a small regularization parameter. Then our reconstruction results are presented in Fig.\,\ref{fig:Inv_heat}.

\begin{figure}[H]
    \centering
        \begin{tabular}{cc}
\includegraphics[trim = {0cm 0.0cm 0cm 0.0cm},clip,scale = 0.475]{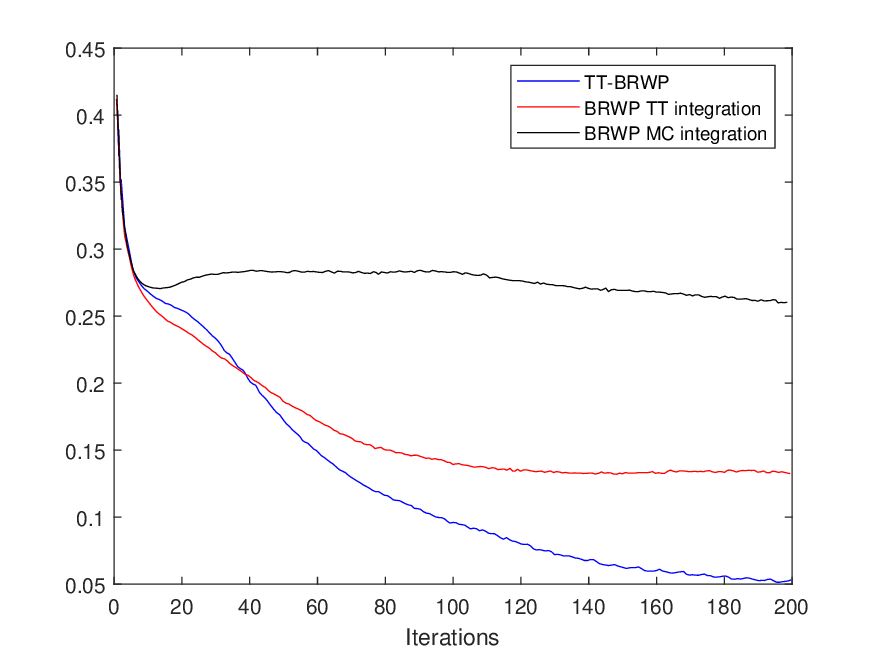}&
\includegraphics[trim = {0cm 0.0cm 0cm 0.0cm},clip,scale = 0.475]{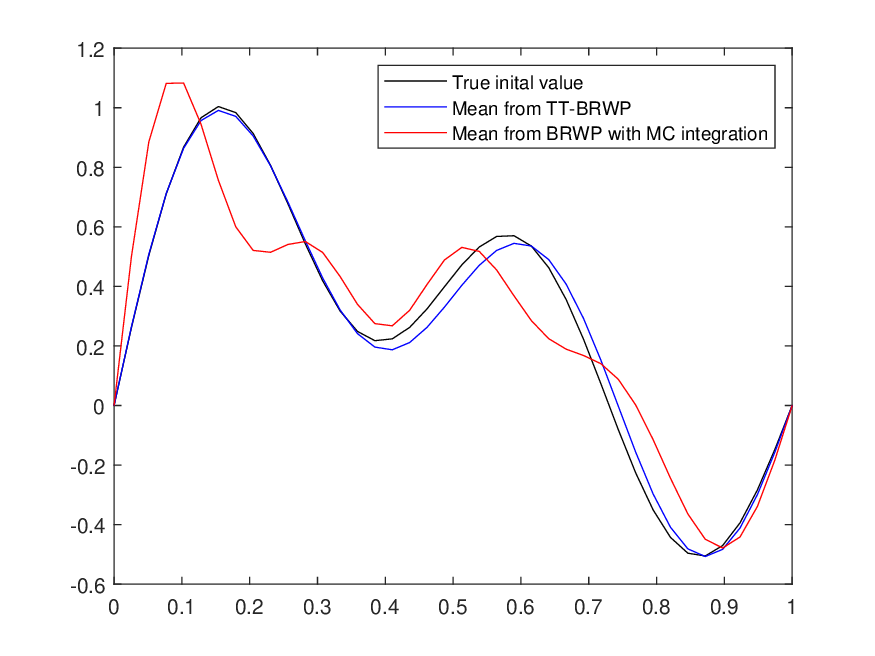}\\
\(L^2\) error between sample mean and true value & Final estimated initial data
\end{tabular}
\caption{Example 7: Recovering initial data for the heat equation under \(L^1\) regularization for \(T = 0.1\) and $h = 0.01$.}
\label{fig:Inv_heat}
\end{figure}

From Fig.\,\ref{fig:Inv_heat}, we can observe that TT-BRWP in blue provides a much better reconstruction compared to sampling algorithms with MC integration in red line and also ULA which does not converge for this task. Moreover, even if we compare BRWP with TT integration (in black) and MC integration (in red) in the first plot, TT integration can indeed improve the accuracy significantly, while the final estimation is still clearly biased due to the employment of empirical distribution for the estimation of the density function.

\noindent\textbf{Example 8:} We consider a nonlinear inverse problem for an elliptic boundary value problem from \cite{Stuart_Bayersian}. The potential $p$ satisfies
\begin{equation}
    -\frac{d}{dy}\bigg(\exp(x_1)\frac{d}{dy}p(y)\bigg) = 1\,,\quad y \in[0,1]\,,
\end{equation}
with $p(0) = 0$ and $p(1) = x_2$. Then the solution to the forward problem has an explicit solution
\begin{equation}
    p(y) = x_2 y + \exp(-x_1)\bigg(-\frac{y^2}{2}+\frac{y}{2}\bigg)\,.
\end{equation}
Due to the exponential term, it is clear that this inverse problem is nonlinear. Then for measurement points $y_1,y_2 \in [0,1]$, the forward operator will be
\begin{equation}
    \mathcal{G}(x) = (p(y_1),p(y_2))^{\ts} \,.
\end{equation}
By employing an $L^2$ regularization term, the potential function will simply be
\begin{equation}
    V(x) = \frac{||\mathcal{G}(x)-\tilde{p}||_2^2}{2\sigma^2} + \tau||x||_2^2\,,
\end{equation}
for noisy measurement data $\tilde{p}$, and noise level $\sigma = 0.1$  which corresponds to $10\%$ noise in the measurement data. In our experiments, we choose $x^*= [0.4,1]$, $y_1 = 0.25$, and $y_2 = 0.75$. 
The reconstruction results for reconstruction error with different choices of step size are presented in Fig.\,\ref{fig:Elliptic}.

\begin{figure}[H]
    \centering
        \begin{tabular}{cc}
\includegraphics[trim = {0cm 0.0cm 0cm 0.0cm},clip,scale = 0.48]{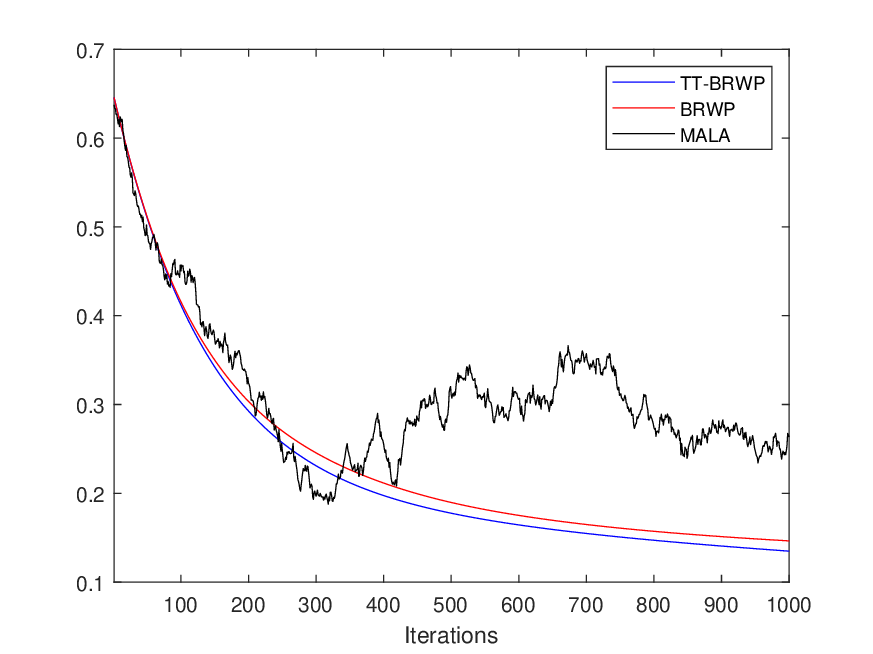}&
\includegraphics[trim = {0cm 0.0cm 0cm 0.0cm},clip,scale = 0.48]{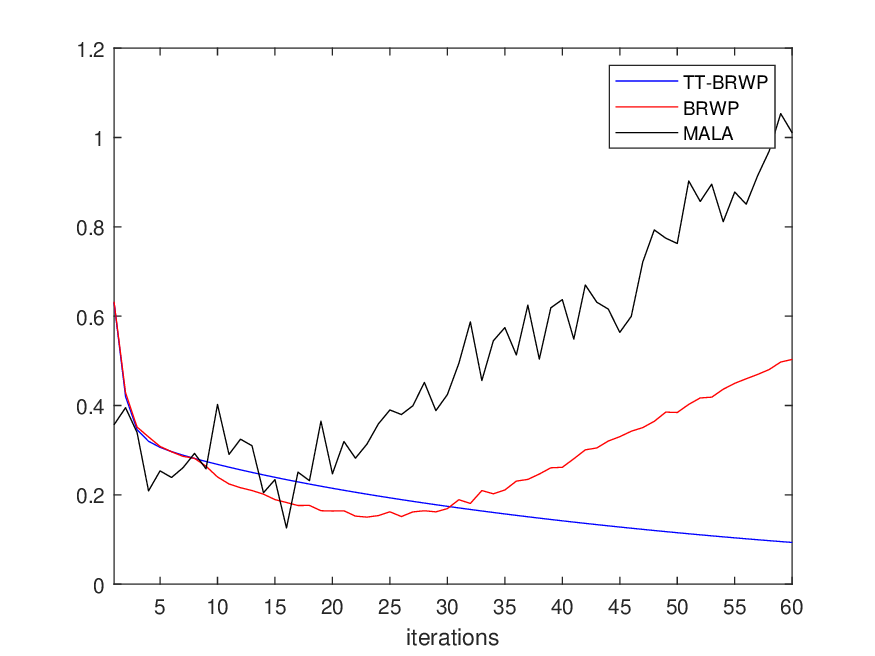}\\
Step size  $h = 10^{-3}$ & Step size $h = 0.1$
\end{tabular}
\caption{Example 8: $L^2$ error of the estimation of parameters $x^*$ versus iteration for an inverse elliptic boundary value problem with different step sizes.}
    \label{fig:Elliptic}
\end{figure}

From Fig.\,\ref{fig:Elliptic}, we observe that when the step size is relatively small, all three methods under consideration exhibit convergence to the stationary distribution, while TT-BRWP in blue provides the most accurate result after a large number of iterations. Moreover, when the step size is large, i.e., $0.1$, in the second plot, TT-BRWP is the only one that provides convergence in a few iterations which demonstrates its strong robustness.

\noindent\textbf{Example 9.} In this example, we examine a nonlinear non-convex inverse problem from \cite{ne_Bayersian_GPT} where we consider the following Cauchy problem
\begin{equation}
        \frac{\partial^2 u}{\partial t^2}(t,x,\theta)- \frac{\partial^2 u}{\partial x^2}(t,x,\theta) = 0\,, \quad
         u(0,x,\theta) = h(x,\theta)\,, \quad u_t(0,x,\theta) = 0\,,
 \end{equation}
 where $h(x,\theta)$ is a unknown source function parameterized by $\theta\in \mathbb{R}^d$ such that
 \begin{equation}
     h(x,\theta) = \sum_{j=1}^d \frac{\sin(k(x-\theta_j))}{k|x-\theta_j|}\,,
 \end{equation}
 where the sinc function can be considered as an approximation to a point source which also makes the problem itself non-convex. Moreover, by d'Alembert's formula, the solution to the Cauchy problem will be
 \begin{equation}
     u(t,x,\theta) = \frac{h(x-t,\theta)+h(x+t,\theta)}{2}\,.
 \end{equation}
The measurement data is considered as $\tilde{u}(x_i,t_i,\theta^*)$ for $x_i$ and $t_j$ distributed uniformly on $[-2,2]$ and $[0,1]$ for $i=1,\cdots,N_i$ and $j = 1,\cdots,N_j$. 

We again employ a $L^2$ regularization term and the potential function we are interested in will be
\begin{equation}
    V(\theta) = \frac{\sum_{i=1}^{N_i}\sum_{j=1}^{N_j}|u(t_j, x_i, \theta)- \tilde{u}(t_j, x_i,\theta^*)|^2}{2\sigma^2}  + \tau ||\theta||_2^2\,,
\end{equation}
with $\sigma = 1$ and $N_i = N_j = 21$. Let $\theta^* = [-0.9,-0.3,0.4,1]$. The result is presented in Fig.\,\ref{fig:wave_source}.

\begin{figure}[H]
    \centering
        \begin{tabular}{cc}
\includegraphics[trim = {0cm 0.0cm 0cm 0.0cm},clip,scale = 0.48]{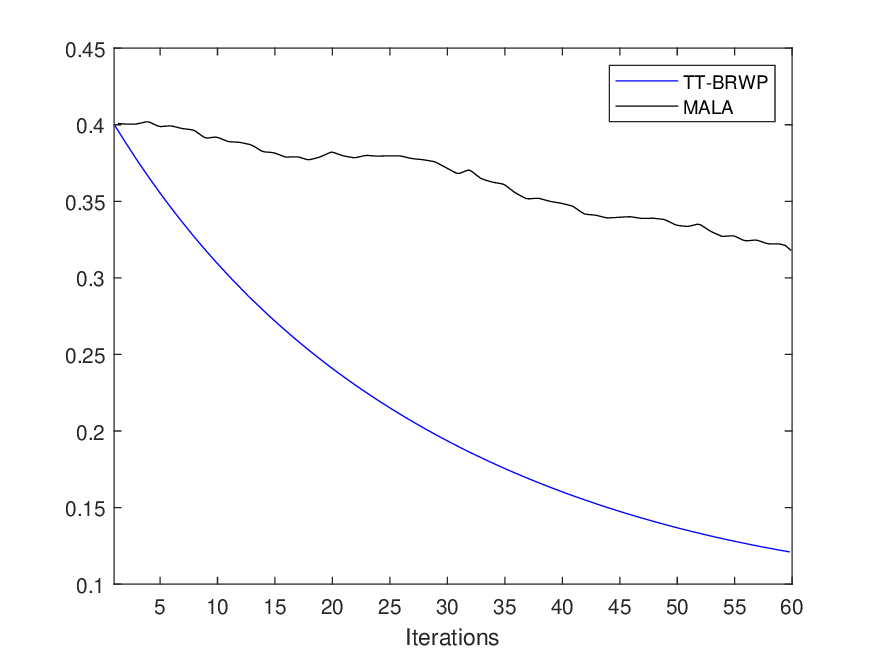}&
\includegraphics[trim = {0cm 0.0cm 0cm 0.0cm},clip,scale = 0.48]{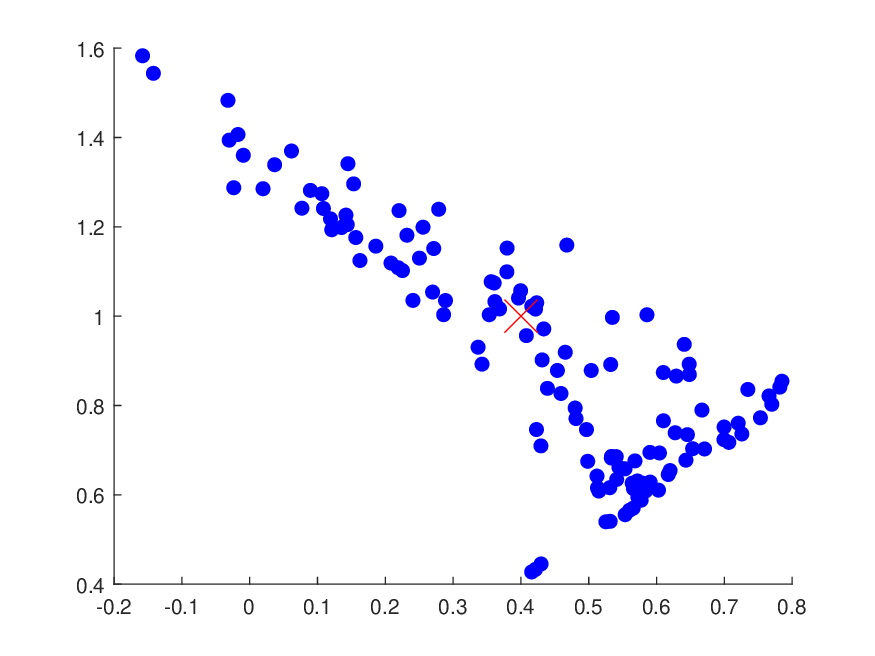}
 \end{tabular}
    \label{fig:wave_source}
\caption{Example 9: Error of the estimation for $\theta$ versus iteration (left) and generated particles from TT-BRWP for $(\theta_3,\theta_4)$ (right) where exact values are $\theta^*_3 = 0.4$, $\theta_4^* = 1$.}
\end{figure}

From Fig.\,\ref{fig:wave_source}, we observe that TT-BRWP converges much faster than MALA to the desired distribution, and the generated particles are distributed near the exact value. We remark that BRWP diverges for this case.

\section{Conclusion.}
In this paper, we proposed a sampling algorithm based on the tensor train approach, aiming to draw samples from potentially high-dimensional and complex distributions. Our method is inspired by a kernel formulation for the regularized Wasserstein proximal operator and employs tensor train approximation for high-dimensional integration. Specifically, by accurately approximating the crucial score function and employing a suitable kernel density approximation, our new sampling algorithm demonstrates superior accuracy, stability, and speed compared to BRWP and Langevin dynamic types sampling algorithms.

\blue{
Compared to the classical ULA and MALA sampling algorithms, the proposed approach offers several attractive features. Firstly, it generates a more structured set of samples due to the diffusion being provided by the score function. Secondly, based on our theoretical analysis and numerical experiments, our method requires fewer iterations to converge. Thirdly, the proposed method demonstrates better stability with larger step sizes because the score function is evaluated at the \( t+T \) time point.}

There are several intriguing avenues for future exploration. From a theoretical standpoint, establishing the accuracy and mixing time of the proposed algorithm for general non-Gaussian target distributions, as verified by our numerical experiments, would be highly attractive. A comprehensive theoretical treatment would enhance the method's applicability to a broader range of practically important scenarios. Moreover, from an algorithmic perspective, exploring the interplay of tensor methods with Wasserstein proximal kernels, investigating the acceleration of MCMC algorithms within the current framework, and exploring related kernel methods for addressing challenging high-dimensional scientific computing problems would be interesting avenues for further research.

\bibliographystyle{siam}
\bibliography{ref.bib}

\end{document}